\documentclass[12pt]{amsart}
\numberwithin{equation}{section}

\usepackage{latexsym, graphicx, amscd, amssymb, xypic,mathdots}
\usepackage{float,color}

\newcommand{\cC}{\mathcal{C}}

\newcommand{\cM}{\mathcal{M}}

\renewcommand{\cR}{\mathcal{R}}
\newcommand{\cS}{\mathcal{S}}

\newcommand{\bC}{\mathbb{C}}

\newcommand{\bR}{\mathbb{R}}
\newcommand{\bZ}{\mathbb{Z}}

\newcommand{\bN}{\mathbb{N}}

\newcommand{\fg}{\mathfrak{g}}
\newcommand{\fk}{\mathfrak{k}}
\newcommand{\fp}{\mathfrak{p}}

\newcommand{\fA}{\mathfrak{A}}

\newcommand{\Ad}{\mathrm{Ad}}

\newcommand{\diag}{\mathrm{diag}}

\newtheorem{thm}{Theorem}[section]

\newtheorem{lm}[thm]{Lemma}

\newtheorem{cor}[thm]{Corollary}
\newtheorem{pro}[thm]{Proposition}

\newtheorem{df}[thm]{Definition}

\newcommand{\tPsiz}{  \tilde\Psi^{(0)}  }
\newcommand{\tPsii}{  \tilde\Psi^{(\infty)}  }
\newcommand{\tpsiz}{  \tilde\psi^{(0)}  }
\newcommand{\tpsii}{  \tilde\psi^{(\infty)}  }
\newcommand{\tSz}{  \tilde S^{(0)}  }
\newcommand{\tSi}{ \tilde S^{(\infty)}  }
\newcommand{\tQz}{  \tilde Q^{(0)}  }
\newcommand{\tQi}{  \tilde Q^{(\infty)}  }

\newcommand{\tE}{  \tilde E  }
\newcommand{\tMz}{  \tilde M^{(0)}  }

\newcommand{\Omz}{  \Omega^{(0)}  }
\newcommand{\Omi}{  \Omega^{(\infty)}  }

\renewcommand{\i}{ {\scriptscriptstyle\sqrt{-1}}\, }

\newcommand{\bp}{\begin{pmatrix}}
\newcommand{\ep}{\end{pmatrix}}
\newcommand{\bsp}{\left(\begin{smallmatrix}}
\newcommand{\esp}{\end{smallmatrix}\right)}

\renewcommand{\sl}{\frak s\frak l}

\newcommand{\h}{\frak h}
\renewcommand{\sec}{ \text{Sect}  }
\newcommand{\supersec}{ \widehat{\text{Sect}}  }
\newcommand{\sto}{ \text{Sto}  }
\newcommand{\Phiz}{  \Phi^{(0)}  }
\newcommand{\Phii}{  \Phi^{(\infty)}  }
\newcommand{\Qz}{ Q^{(0)}  }

\newcommand{\simple}{\Gamma}

\newcommand{\n}{\sigma}

\begin{document}

\title[The tt*-Toda equations]{A Lie-theoretic description of the solution space
\\
of the tt*-Toda equations}

\author{Martin A. Guest}
\address{Department of Mathematics, Waseda University, 3-4-1 Okubo, Shinjuku, Tokyo 169-8555, Japan}
\email{martin@waseda.jp}

\author{Nan-Kuo Ho}
\address{Department of Mathematics, National Tsing Hua University, Hsinchu 300, and National Center for Theoretical Sciences, Taipei 106, Taiwan}
\email{nankuo@math.nthu.edu.tw}

\date{\today}

\maketitle

\begin{abstract}

We give a Lie-theoretic explanation for the convex polytope which parametrizes the globally smooth solutions of the topological-antitopological fusion equations of Toda type
(tt$^*$-Toda equations) which were introduced by Cecotti and Vafa.  It is known
from \cite{GL}\cite{GIL1}\cite{M1}\cite{M2}
that these solutions can be parametrized by monodromy data of a certain flat $SL_{n+1}\bR$-connection.  Using Boalch's Lie-theoretic description of  Stokes data, and Steinberg's description of regular conjugacy classes of a linear algebraic group, we express this monodromy data as a convex subset of a Weyl alcove of $SU_{n+1}$.
\end{abstract}

\section{Introduction}\label{intro}

The tt$^*$ equations were introduced 25 years ago by Cecotti and Vafa \cite{CV1} in order to describe deformations of supersymmetric quantum field theories, and interpreted in the mathematical context of integrable systems by Dubrovin \cite{Du93}.  They remain highly visible in current
developments in physics \cite{GMN}\cite{CGV} --- as does the unfortunate fact that mathematical results concerning their solutions are few and far between.

In their original work Cecotti and Vafa discussed at length a special
case of the tt$^*$ equations
which is related to the periodic Toda field equations, namely
\begin{equation}\label{ost}
2\tfrac{\partial^2}{\partial t \partial \bar{t}}w_i=-e^{2w_{i+1}-2w_i}+e^{2w_i-2w_{i-1}}
\end{equation}
where
$w_i:\bC^*\to\bR$ and $w_i$ depends only on
$\vert t\vert$.  Here $i\in\bZ$, and it is assumed that the $w_i$ satisfy
$w_{n-i}+w_i=0$ as well as the periodicity condition
$w_{i+n+1}=w_{i}$.
For $n=1$ these \lq\lq  tt$^*$-Toda equations\rq\rq\ reduce to the radial sinh-Gordon equation (a case of the third Painlev\'e equation). In this case mathematical results were available, and used, in \cite{CV1}.
On the other hand, the solutions of the tt$^*$-Toda equations for other $n$ were obtained relatively recently, in \cite{GL}\cite{GIL1}\cite{M1}\cite{M2}.  It was shown, as Cecotti and Vafa had predicted, that the solutions are characterized by their \lq\lq boundary conditions\rq\rq, more precisely by their asymptotics at
$t=0$ or at $t=\infty$.

Although the tt$^*$-Toda equations are indeed quite special, the existence of these mathematical results makes it possible to probe more deeply, in order to  verify further physical predictions and possibly gain insight into more general cases of the tt*-equations.   In the present article we focus on the Lie-theoretic structure of the space of solutions.
It is of course well known that the Toda equations have a Lie-theoretic origin, and for the above equations the appropriate Lie group is just $SL_{n+1}\bR$.  However the Lie-theoretic structure studied in this article appears in quite a subtle way, from the monodromy data (in particular, Stokes matrices) of a related system of meromorphic o.d.e.
As far as we know, this Lie-theoretic aspect has not been studied before, although (as we shall see in section \ref{localsolutions}) it is related to the \lq\lq unitarity implies regularity\rq\rq\ principle of \cite{CV2}, section 6.

Our results can be summarized as follows.
From \cite{GL}\cite{GIL1}\cite{M1}\cite{M2} it is known that the solutions are parametrized by their asymptotic data at $t=0$, or by certain entries of Stokes matrices, and this data can be identified with a certain convex polytope.
In section \ref{tt*toda} we extend some results of \cite{GIL1} for $n=3,4$ (cases 4a, 5a) to general $n$; this general case was also treated in \cite{M1} in rather different notation.
In section \ref{boalch} we apply the theory developed by Boalch \cite{B_imrn} to express these Stokes matrices in terms of the roots of the complex Lie group $SL_{n+1}\bC$
(Propositions \ref{pro:stokesfactoreven}, \ref{pro:stokesfactorodd}). Then we observe that --- for the tt$^*$-Toda equations --- the Stokes matrices are determined by certain special entries, which correspond to a set of {\em simple} roots
(Propositions \ref{pro:simplerootseven}, \ref{pro:simplerootsodd}).  Next,  in
section \ref{steinberg}, using classical  results of Steinberg, we show that these entries parametrize certain conjugacy classes of elements of $SL_{n+1}\bC$ (Theorem \ref{thm:m-is-r}).  Finally  in
section \ref{localsolutions}
we show that the set of conjugacy classes which actually come from solutions of the tt$^*$-Toda equations can be identified with a set of conjugacy classes of elements of the {\em compact} group $SU_{n+1}$ (Theorem \ref{thm:main}), which in turn can be identified with a convex subset of the Weyl alcove of $SU_{n+1}$ --- in fact the fixed points of a specific involution.
Thus, the convex polytope which arises from solving the tt$^*$-Toda equations has been given a natural Lie-theoretic explanation (Proposition \ref{pro:lietheoretic}).

For readers familiar with the Toda equations, we would like to emphasize that the compact group
$SU_{n+1}$ does not arise directly from the equations themselves.  There is indeed a version of the Toda equations corresponding to $SU_{n+1}$, but this is the equation obtained from (\ref{ost}) by changing the sign of the right hand side.  Its solutions have quite different properties to those of (\ref{ost}).  They have been much-studied by differential geometers as examples of harmonic maps into compact homogeneous spaces (cf.\ \cite{Mc}\cite{BPW}).  In contrast, the group $SU_{n+1}$ appears from (\ref{ost}) in a very indirect way, from part of the monodromy data in the (irregular) Riemann-Hilbert correspondence.

Readers familiar with Lie theory will not be surprised to learn that
most of our results can be generalized to the situation where $SL_{n+1}\bC$ is replaced by any complex simple Lie group, and we shall discuss this in \cite{GH}. We prefer to treat the case $G=SL_{n+1}\bC$ separately here in a down to earth and explicit fashion, using the notation of \cite{GL}\cite{GIL1}.  For the general case it is more convenient to make certain choices (for example, of Stokes data) in a different and more abstract way.

Acknowledgements:  The authors thank Philip Boalch, Ralf Kurbel, and Eckhard Meinrenken  for useful conversations.  The first author was partially supported by JSPS grant (A) 25247005. He is grateful to the National Center for Theoretical Sciences for excellent working conditions and financial support. The second author was partially supported by MOST grant 104-2115-M-007-004.  She is grateful to the Japan Society for the Promotion of Science for the award of a Short-Term Invitation Fellowship.

\section{The tt$^*$-Toda equations}\label{tt*toda}

Let $n\in\bN$.
We consider real valued smooth functions $w_i:\bC^*\to\bR$ with $w_{i+n+1}=w_{i}$
and $w_{n-i}+w_i=0$ for all $i\in\bZ$.
Let
\[
 \alpha = (w_t+\tfrac1\lambda W^T)dt + (-w_{\bar t}+\lambda W)d\bar t,
\]
where
 \[
 w=\diag(w_0,\dots,w_n),\ \
 W=
 \left(
\begin{array}{c|c|c|c}
\vphantom{(w_0)_{(w_0)}^{(w_0)}}
 & \!e^{w_1\!-\!w_0}\! & &
 \\
\hline
  &  & \  \ddots\   & \\
\hline
\vphantom{\ddots}
  & &  &  e^{w_n\!-\!w_{n\!-\!1}}\!\!\!
\\
\hline
\vphantom{(w_0)_{(w_0)}^{(w_0)}}
\!\! e^{w_0\!-\!w_n} \!\!  & &  &  \!
\end{array}
\right),
\]
$w_t$ means $\partial w/\partial t$,
and $W^T$ is the transpose of $W$.  Here $t$ is the coordinate of $\bC$, and $\lambda$ is a nonzero complex parameter.

We regard $d+\alpha$ as a $\lambda$-family of  connections in the trivial complex bundle over $\bC^\ast$ of rank $n+1$.  The condition that $d+\alpha$ is flat for all $\lambda$ is equivalent to
$2w_{t\bar t} = [W^T,W]$, i.e.\
\begin{equation}\label{toda}
2(w_i)_{t\bar t}=-e^{2(w_{i+1}-w_{i})} + e^{2(w_{i}-w_{i-1})}\ \text{for all $i$},
\end{equation}
which is a version of the Toda field equations.

On the other hand, if we impose the condition $w_i=w_i(\vert t\vert)$
then we have a special case of the topological-antitopological fusion or tt$^*$ equations which were introduced by Cecotti and Vafa in \cite{CV1}.
For this reason, Guest-Its-Lin \cite{GIL1}\cite{GIL2} called (\ref{toda}) together with the radial condition
$w_i=w_i(\vert t\vert)$  the {\em tt$^*$-Toda equations.}

\begin{thm}\cite{GL}\cite{GIL1}\cite{GIL2}\cite{M1}\cite{M2}\label{thm:convex}
There is a bijection between the set of solutions on $\bC^*$ to the $tt^*$-Toda equations $w_0,\dots,w_n$ and the bounded convex region
\[
\{(\gamma_0,\dots,\gamma_n)\in\bR^{n+1}\mid
\gamma_{i+1}-\gamma_i\geq -2, ~\gamma_{n-i}+\gamma_i=0\}.
\]
The correspondence is given by $2w_i(t)\sim \gamma_i\log |t|$, as $|t|\to 0$.
\end{thm}

(The notation $f(t)\sim g(t)$, $|t|\to 0$ means $\lim_{\vert t\vert\to 0} f(t)/g(t)=1$, cf.\ \cite{FIKN06} Chapter 1, section 0.  The asymptotics of solutions are explained in more detail in \cite{GIL1}\cite{GIL2} --- see in particular the introduction to \cite{GIL2}.)

We shall give a Lie-theoretic explanation for the appearance of this convex set.  Although the inequalities $\gamma_{i+1}-\gamma_i\geq -2$ might already suggest a relation with the Weyl alcove of $SU_{n+1}$, the role of this group is not apparent from the definition of $\alpha$.  Indeed, the group-theoretic properties of $\alpha$ arise from the following \lq\lq symmetries\rq\rq, none of which relate to compact groups:

\noindent{\em Cyclic symmetry: }  $\tau(\alpha(\lambda))=\alpha(\omega \lambda)$

\noindent{\em Anti-symmetry: }  $\sigma(\alpha(\lambda))=\alpha(-\lambda)$

\noindent{\em Reality: }  $c(\alpha(\lambda))=\alpha(1/\bar\lambda)$

\noindent Here $\tau, \sigma, c$ are the automorphisms of $\sl_{n+1}\bC$ defined by
\[
\tau(X)=d_{n+1}^{-1} X d_{n+1},\ \ \sigma(X)=-\Delta\, X^T\,  \Delta,\ \ c(X)=\Delta \bar X  \Delta
\]
where
\[
d_{n+1}
=
\bp
1 & & & \\
 & \omega & & \\
 & & \ddots & \\
 & & & \omega^n
\ep,
\quad
\Delta
=
\bp
  & & 1 \\
  & \iddots \, & \\
1 & &
\ep
\]
and $\omega=e^{{2\pi \i}/{(n+1)}}$.
These properties of $\alpha$ are easily verified, and the reality condition shows that $\alpha$ takes values in the Lie algebra
\[
\sl_{n+1}^\Delta\bR = \{ X\in \sl_{n+1}\bC \ \vert\  \Delta \bar X\Delta = X \}
\]
when $\lambda = 1/\bar\lambda$, i.e.\ $\vert\lambda\vert=1$.  As explained in section 2 of \cite{GL}, we have
$\sl_{n+1}^\Delta\bR \cong \sl_{n+1}\bR$, and also $SL_{n+1}^\Delta\bR \cong SL_{n+1}\bR$, where
$SL_{n+1}^\Delta\bR = \{ X\in SL_{n+1}\bC \ \vert\  \Delta \bar X\Delta = X \}$.
In this sense $d+\alpha$ is an $SL_{n+1}\bR $-connection with additional symmetries given by
$\tau$ and $\sigma$.

\section{Stokes Analysis}\label{boalch}

In this article we shall focus on another form of the radial version of (\ref{toda}), related to a $\vert t\vert$-family of  (flat) meromorphic connections $d+\hat\alpha$ on the trivial complex vector bundle of rank $n+1$ over the Riemann sphere $\bC P^1=\bC\cup\infty$, with coordinate $\zeta$ on $\bC$. This is given by
\[
\hat{\alpha}(\zeta)=\left(
-\tfrac{1}{\zeta^2}W^T-\tfrac{1}{\zeta}xw_x+x^2W
\right)
d\zeta
\]
where $x=|t|$.  Here $W$ is as above, with $w_i=w_i(\vert t\vert)$, and $\zeta=\lambda/t$,
but we refer to section 2 of \cite{GIL3} for the precise relation between
$\hat\alpha$ and $\alpha$.  The essential point is that $d+\hat\alpha$ is isomonodromic (i.e.\ its monodromy data at the poles $\zeta=0,\infty$ is independent of $x$) if and only if
$(xw_{x})_x=2x[W^T,W]$, which is the radial version of (\ref{toda}).

The symmetries of $\alpha$ translate into the following symmetries of $\hat{\alpha}$:

\noindent{\em Cyclic symmetry: }  $\tau(\hat\alpha(\zeta))=\hat\alpha(\omega \zeta)$

\noindent{\em Anti-symmetry: }  $\sigma(\hat\alpha(\zeta))=\hat\alpha(-\zeta)$

\noindent{\em Reality: }  $c(\hat\alpha(\zeta))=\hat\alpha(1/(x^2\bar\zeta))$

\noindent and for $\hat\alpha$ there is a further reality property $\overline{\hat{\alpha}(\bar{\zeta})}=\hat{\alpha}(\zeta)$.

As in the case of $\alpha$, these properties show that $\hat\alpha$ takes values in the Lie algebra
$\sl_{n+1}^\Delta\bR \cong \sl_{n+1}\bR$
when $\vert\zeta\vert=1/x$.

The monodromy data of $d+\hat\alpha$ consists of

(i) formal monodromy matrices at the poles $\zeta=0,\infty$

(ii) Stokes matrices (relating flat holomorphic sections on Stokes sectors at each pole)

(iii) connection matrices (relating flat  holomorphic sections at $\zeta=0$ to flat  holomorphic sections at
$\zeta=\infty$).

\noindent It is a fundamental principle (the Riemann-Hilbert correspondence) that local isomonodromic deformations --- i.e.\ local solutions of (\ref{toda}) --- correspond to such monodromy data
(statements which apply to $d+\hat\alpha$ can be found in
\cite{B_adv},\cite{B_duke}).   In this section and the next, however, we just discuss the structure of the monodromy data itself, for which it suffices to consider $t$ fixed.  We return to the consideration of solutions of (\ref{toda}) in section \ref{localsolutions}.

We shall recall the description of the monodromy data from \cite{GIL1}\cite{GIL2}, then express it in the language of \cite{B_inv},\cite{B_imrn}.
It will be convenient to consider separately the cases where $n+1$ is even or odd.  To avoid introducing new notation for the case $n=1$ we assume $n\ge 2$ from now on.

\subsection{Monodromy data for the case $n+1=2m$}
$  $

We express the monodromy data in the (classical) context of the equation $(d-\hat\alpha^T)\Psi=0$, referring to \cite{FIKN06} for general facts about such equations (see also \cite{BJL}).  This notation means that the columns of the matrix function $\Psi$ are a basis of flat sections of the dual connection $d-\hat\alpha^T$. Explicitly, the equation is
\begin{equation}\label{meromorphicsystem}
\Psi_\zeta =  \left( -\tfrac{1}{\zeta^2} W - \tfrac{1}{\zeta} xw_x + x^2W^T \right) \Psi.
\end{equation}
The following presentation of the monodromy data generalizes that for $n+1=4$ in section 4 of \cite{GIL1} and section 5 of \cite{GIL2}.

\noindent{\em Choice of formal solutions.}

To diagonalize the coefficient $-W$ of $1/\zeta^2$, we use the fact that
\[
W=e^{-w\, }\Pi\,  e^w,
\quad
\Pi=
\begin{pmatrix}
  & 1 & & \\
 & & \ \ddots & \\
  & & & 1\\
1   & & &
\end{pmatrix}.
\]
The eigenvalues of $\Pi$ are the $(n+1)$-th roots of unity, and in fact we have
$\Pi= \Omega \, d_{n+1}\,  \Omega^{-1}$, where
$\Omega$ is the Vandermonde matrix
$(\omega^{ij})_{0\le i,j\le n}$.  We obtain
\[
-W = \tilde P_0 \,(-d_{n+1}) \, \tilde P_0^{-1}
\]
where, following section 5 of \cite{GIL2}, we choose $\tilde P_0=e^{-w}\, \Omega\,
\diag(1,\omega^{\frac12},\omega^{1},\omega^{\frac32},\dots,\omega^{\frac{n+1}2})$.

Having made the specific choices $-d_{n+1}$ (ordering of the eigenvalues) and $\tilde P_0$ (diagonalizing matrix), by Proposition 1.1 of \cite{FIKN06} we obtain a unique formal solution of the form
\[
\tPsiz_f(\zeta) = \tilde P_0\left( I + \sum_{k\ge 1} \tpsiz_k \zeta^k \right) e^{\frac1\zeta  d_{n+1}}.
\]
Similarly, we have $W^T=e^{w}\,\Pi^T\, e^{-w}$, and $\Pi^T= \Omega^{-1} \, d_{n+1}\,  \Omega$.  We obtain
\[
x^2W^T = \tilde P_\infty \,(x^2 d_{n+1}) \, \tilde P_\infty^{-1}
\]
by choosing $\tilde P_\infty=e^{w}\, \Omega^{-1}\,
\diag(1,\omega^{-\frac12},\omega^{-1},\omega^{-\frac32},\dots,\omega^{-\frac{n+1}2})$.
Then there is unique formal solution at $\zeta=\infty$ (i.e.\ $1/\zeta=0$)  of the form
\[
\tPsii_f(\zeta) = \tilde P_\infty\left( I + \sum_{k\ge 1} \tpsii_k \zeta^{-k} \right) e^{x^2 \zeta  d_{n+1}}.
\]

\noindent{\em Stokes matrices.}

Let us take initial Stokes sectors at zero and infinity to be
\[
\Omz_1=\{ \zeta\in\bC^\ast \ \vert \ -\tfrac{\pi}{2}-\tfrac{\pi}{n+1}<\text{arg}\zeta <\tfrac\pi2\},
\quad
\Omi_1=\{ \zeta\in\bC^\ast \ \vert\  -\tfrac\pi2<\text{arg}\zeta < \tfrac{\pi}{2}+\tfrac{\pi}{n+1}\}.
\]
We regard these as subsets of
the universal covering surface
$\tilde\bC^\ast$, and define further Stokes sectors in $\tilde\bC^\ast$ by
\[
(\Omz_{k+\frac{1}{n+1}}=)\
\Omz_{k\frac{1}{n+1}}= e^{-\frac\pi{n+1}\i} \Omz_k,
\quad
\Omi_{k\frac1{n+1}}= e^{\frac\pi{n+1}\i} \Omi_k
\]
for any $k\in\tfrac1{n+1}\bZ$.
Then
(see, for example, Theorem 1.4 of \cite{FIKN06})
there are unique holomorphic solutions
$\tPsiz_k,\tPsii_k$ on $\Omz_k, \Omi_k$
such that
$\tPsiz_k\sim\tPsiz_f$ as $\zeta\to0$ in $\Omz_k$ and
$\tPsii_k\sim\tPsii_f$ as $\zeta\to\infty$ in $\Omi_k$.
We use the same notation $\tPsiz_k,\tPsii_k$ for the analytic continuations of these solutions to the whole of $\tilde\bC^\ast$.

The symmetries of $\hat\alpha$ lead to the following symmetries of the $\tPsiz_k$ (see section 2 of \cite{GIL2} for these formulae, and Appendix A of
\cite{GIL2} for a list of commonly used notation):
\begin{itemize}
\item
$d_{n+1}^{-1}\tPsiz_{k-\scriptstyle\frac2{n+1}}(\omega\zeta)\hat\Pi^{-1}\omega^{-\frac12}=\tPsiz_k(\zeta),
\
\hat\Pi=
\left(
\begin{smallmatrix}
  &  1  & & &\\
  &   & \ \ddots & &\\
  & & &  & 1\\
-1    & & &
\end{smallmatrix}
\right)
$
\item
$(n+1)\Delta \left(\tPsiz_{k-1}(e^{\i\pi}\zeta)^{T}\right)^{-1}=\tPsiz_k(\zeta)$
\item
$(n+1)\Delta \tPsii_{k}(\frac 1{x^2\zeta}) = \tPsiz_{k}(\zeta)$
\item
$\overline{  \tPsiz_{\scriptstyle\frac{2n+1}{n+1}-k}(\bar\zeta)  } \tilde C = \tPsiz_k(\zeta),
\
\tilde C=
\left(
\begin{smallmatrix}
1 & & & &\\
 & & & & -1\\
 & & & \iddots &\\
 & -1 & & &
\end{smallmatrix}
\right)$.
\end{itemize}

On overlapping sectors these solutions differ by constant matrices (independent of $\zeta$ and $x$).  We define specific matrices $\tQz_k, \tQi_k$ by
\[
\tPsiz_{k\scriptstyle\frac1{n+1}} = \tPsiz_k \tQz_k,
\quad
\tPsii_{k\scriptstyle\frac1{n+1}} = \tPsii_k \tQi_k
\]
and $\tSz_k, \tSi_k$ by
\[
\tPsiz_{k+1}=\tPsiz_k \tSz_k,
\quad
\tPsii_{k+1}=\tPsii_k \tSi_k.
\]
It follows that
\[
\tSz_k=\tQz_{k}\tQz_{k\scriptstyle\frac1{n+1}}\tQz_{k\scriptstyle\frac2{n+1}}\dots
\tQz_{k\scriptstyle\frac{n}{n+1}},
\quad
\tSi_k=\tQi_{k}\tQi_{k\scriptstyle\frac1{n+1}}\tQi_{k\scriptstyle\frac2{n+1}}\dots
\tQi_{k\scriptstyle\frac{n}{n+1}}.
\]
From the symmetries of the $\tPsiz_k$ we obtain:
\begin{itemize}
\item
$\tQz_{k+\scriptstyle\frac2{n+1}} =
\hat\Pi\  \tQz_k \  \hat\Pi^{-1}
$
\item
$\tQz_{k+1} = \left(  (\tQz_k)^{T}\right)^{-1}$
$\vphantom{\dfrac12}$
\item
$\tQz_{k}=\tQi_{k}$
$\vphantom{\dfrac12}$
\item
$\tQz_k= \tilde C \
\left( \, \overline{\tQz_{ {\scriptstyle\frac{2n}{n+1} - k}  }}\, \right)^{-1} \ \tilde C$.
\end{itemize}
In particular, the first (cyclic) symmetry shows that all $\tQz_k$ are determined by $\tQz_1,\tQz_{1\scriptstyle\frac1{n+1}}$. The third shows that it suffices to consider the pole $\zeta=0$.

We can say more about the shape of the matrices $\tQz_k$:

\begin{lm}\label{lm:shapeofQeven}
When $n+1=2m$, the $(i,j)$ entry of the matrix $\tQz_{1\scriptstyle\frac\ell{n+1}}$ satisfies
\[
\left(
\tQz_{1\scriptstyle\frac\ell{n+1}}
\right)_{i,j}=
\begin{cases}
1 \quad \text{if}\ i=j
\\
0 \quad
\text{if $i\ne j$, and
$\arg(\omega^i-\omega^j) \not\equiv
(n-\ell)\tfrac{\pi}{n+1}$ mod $2\pi$
}
\end{cases}
\]
\end{lm}

\noindent This is proved in exactly the same way as Lemma 4.4 of \cite{GIL1} (the case $n+1=4$).  That lemma applies to the matrices $Q_k$ (defined in section 4 of \cite{GIL1}), but in fact those matrices have the same shape as the $\tQz_k$.
The remaining entries depend on $w_0,\dots,w_n$ (i.e.\ the coefficients of $\hat\alpha$), and will be given later (section \ref{localsolutions}).  The reason for using $\tQz_k$ rather than $Q_k$ is that all entries of $\tQz_k$ are real numbers (see Corollary \ref{cor:real}).

\noindent{\em Monodromy matrices.}

As in the proof of Lemma 4.2 of \cite{GIL1}, the triviality of the formal monodromy gives
\[
\tPsiz_1(e^{2\pi\i}\zeta)=\tPsiz_1(\zeta) \tSz_1\tSz_2,
\]
i.e.\ the monodromy of $\tPsiz_1$
(parallel translation around the unit circle $\vert\zeta\vert=1$ in the positive direction)
is $\tSz_1\tSz_2$.  From the cyclic symmetry property we obtain
\[
\tSz_1\tSz_2
=-(\tQz_1\tQz_{1\scriptstyle\frac1{n+1}}\hat\Pi)^{n+1}
=(\tQz_1\tQz_{1\scriptstyle\frac1{n+1}}\hat\Pi \,\omega^{\frac12})^{n+1}.
\]

\noindent{\em Connection matrices.}

Solutions at zero and infinity are related in a similar way, by \lq\lq connection matrices\rq\rq\ $\tE_k$, namely
\[
\tPsii_k=\tPsiz_k\tE_k.
\]
The final constituents of the monodromy data  are these connection matrices $\tilde E_k$, although we shall not need explicit formulae for them in this article.  They are all determined by $\tilde E_1$.

As remarked earlier, there is a correspondence between local isomonodromic deformations (or solutions of the tt*-Toda equations) and the monodromy data $\tilde E_1,\tQz_1,\tQz_{1\scriptstyle\frac1{n+1}}$.  (In general one would have formal monodromy matrices as well, but these are trivial in our situation.)
It turns out that all the global solutions on $\bC^\ast$ have the same value of $\tilde E_1$ (see Theorem 5.5 of \cite{GIL2}). Therefore, since these global solutions are our main focus,
only the data $\tQz_1,\tQz_{1\scriptstyle\frac1{n+1}}$ will be needed, and, as we shall see in the next section, giving these is equivalent to giving the single matrix $\tQz_1\tQz_{1\scriptstyle\frac1{n+1}}\hat\Pi$.  With this in mind, we introduce the following notation:

\begin{df}\label{def:qqpi}
When $n+1=2m$, $\tMz
\overset{\scriptstyle \text{\em def} }=
\tQz_1\tQz_{1\scriptstyle\frac1{n+1}}\hat\Pi$.
\end{df}

We emphasize that $\tMz$ is constructed from a particular meromorphic connection $\hat\alpha$, i.e.\ from a particular solution of the tt*-Toda equations, even though this dependence is not indicated explicitly in the notation.  Later we shall investigate which matrices $\tMz$ actually arise from solutions.

As we have just seen, the $(n+1)$-th power of $\tMz\omega^{\frac12}$ is equal to the monodromy of
$\tPsiz_1$.  Geometrically, this can be explained as follows.  First, the cyclic symmetry property says that the parallel translation of a (holomorphic)  flat section $\Psi$ along the unit circle from $\zeta$ to
$\omega\zeta$, followed by $d_{n+1}^{-1}$, is a well-defined automorphism $M$ of the fibre at $\zeta$,  which satisfies $d_{n+1}^{-1}\Psi(\omega\zeta)=\Psi(\zeta)M$.   In fact, the formula
\[
d_{n+1}^{-1}\tPsiz_1(\omega\zeta)
=\tPsiz_1(\zeta)
\
\tQz_1 \tQz_{1\scriptstyle\frac1{n+1}}\hat\Pi \,\omega^{\frac12}
=
\tPsiz_1(\zeta)
\
\tMz\omega^{\frac12}
\]
shows that, in the case $\Psi=\tPsiz_1$, we have
$M=\tMz\omega^{\frac12}$.  For this reason, we regard $\tMz$ as the fundamental monodromy data for the connection $\hat\alpha$.

Next, we will review the framework of Boalch \cite{B_imrn}, which describes the Stokes matrices of meromorphic $G$-connections efficiently in Lie-theoretic terms, for any connected complex reductive group $G$.  Even though we consider only the group $G=SL_{n+1}\bC$ here, we shall find this notation useful.

Let us choose the standard Cartan subalgebra
\[
\h_{n+1}=
\{
\diag(h_0,\dots,h_n)\in\sl_{n+1}\bC
\ \vert\
h_i\in\bC\}
\]
of $\sl_{n+1}\bC$.  For $i=0,\dots,n$ we define
$v_i\in\h_{n+1}^\ast$ by
$v_i(\diag(h_0,\dots,h_n))=h_i$.
Then the roots for $\sl_{n+1}\bC$ with respect to $\h_{n+1}$
are $\alpha_{i,j}=v_i-v_j$ with $0\leq i \neq j \leq n$.  We denote
the set of all roots by $\Delta_{n+1}$.  We have the root space decomposition
\[
\sl_{n+1}\bC=\h_{n+1}\ \oplus\ \left(\oplus_{\alpha\in\Delta_{n+1}} \bC E_{i,j}\right)
\]
where $E_{i,j}$ is the matrix with $1$ in the $(i,j)$ entry and all other entries zero.

The Stokes data of \cite{B_imrn} is based on a choice of ordering of the singular directions (anti-Stokes directions) $d_1,d_2,\dots,d_{2l}$ associated to a pole of a meromorphic connection.
In the case of a pole of order $2$, with leading term $\tfrac1{\zeta^2}X$, the singular directions are simply the directions of (or rays through) the complex numbers $\alpha(X)$, where
$\alpha\in\Delta_{n+1}$. Here it is assumed that all $\alpha(X)$ are nonzero.
Given a singular direction $d$, the set of roots supported by $d$, i.e.\ for which $\alpha(X)$ has direction $d$, is denoted by $\cR(d)$.

In the terminology of \cite{B_imrn}, the $i$-th sector $\sec_i$ is the sector between $d_i,d_{i+1}$.  The $i$-th supersector $\supersec_i$ is the larger sector bounded by $\arg d_i-\tfrac\pi2,\arg d_{i+1}+\tfrac\pi2$. On $\supersec_i$ there is a unique holomorphic solution $\Phi_i$ of (\ref{meromorphicsystem}) asymptotic to a chosen formal solution $\Phi_f$ as $\zeta\to 0$ in that sector.  The {\em Stokes factors} $K_2,\dots,K_{2l}$ are defined by  $\Phi_i=\Phi_{i+1} K_{i+1}$ (and $K_1$ is defined by $\Phi_{2l-1} M_0^{-1}=\Phi_{1} K_{1}$ where $M_0$ is the formal monodromy).  Thus, each singular direction $d_i$ has an associated Stokes factor $K_i$.  One has the following result concerning the shape of $K_i$:

\begin{lm}\label{lm:shapeofQevenBOALCH} \cite{B_imrn} Let $U_\alpha$ be the (unipotent) subgroup of $SL_{n+1}\bC$ corresponding to the root $\alpha$, i.e.\
$U_{\alpha_{i,j}}=\exp(\bC E_{i,j})$.  Then
$K_i$ belongs to the {\em group of Stokes factors} $\sto_{d_i} =
\overset{\scriptstyle }{\underset{\alpha\in\cR(d_i)}\Pi} U_\alpha$ (this is independent of the ordering of the factors).
\end{lm}

In our situation $X=\diag(-1,-\omega,\dots,-\omega^n)$, from which it follows that the arguments of the singular directions are $\tfrac\pi{n+1}\bZ$ mod $2\pi$ (and $l=2n+2$). In \cite{B_imrn}  the $d_i$ are regarded as the unit complex numbers $\alpha(X)/\vert \alpha(X)\vert$ and all sectors are in $\bC^\ast$.
In keeping with the conventions of  \cite{GIL1}\cite{GIL2}, however, we shall replace  the singular directions $d$ by their arguments $\theta$ in $\bR$, and regard sectors as in the universal cover $\tilde\bC^\ast$.  As in \cite{GIL1}\cite{GIL2} we continue to use
$k\in\tfrac 1{n+1}\bZ$ as labels for all objects.

Let us choose initial singular direction
$\theta_1 = -\tfrac\pi{n+1}$,
and denote further singular directions by
\[
\theta_{1\scriptstyle\frac \ell{n+1}} = -\tfrac{\ell+1}{n+1}\pi,\quad \ell\in\bZ.
\]
Projecting to $\bC^\ast$ for $\ell=0,\dots,2n+1$ gives the singular directions in the sense of \cite{B_imrn}.  Apart from the (unimportant) fact that our labels increase in the clockwise direction, we are now in the situation of \cite{B_imrn}, as
\[
\Omz_{1\scriptstyle\frac \ell{n+1}}=
\{ \zeta\in\bC^\ast \ \vert\  \theta_{1\scriptstyle\frac\ell{n+1}}-\tfrac\pi2<\text{arg}\zeta < \theta_{1\scriptstyle\frac{\ell-1}{n+1}}+\tfrac{\pi}{2}\}
=\supersec_{1\scriptstyle\frac \ell{n+1}}.
\]
More precisely, this holds (by definition) for the initial sector given by $\ell=0$, then the projection of $\Omz_{1\scriptstyle\frac \ell{n+1}}$ to $\bC^\ast$ is $\supersec_{1\scriptstyle\frac \ell{n+1}}$ for all $\ell$.
It follows that
\[
\tQz_{1\scriptstyle\frac \ell{n+1}}=K_{1\scriptstyle\frac \ell{n+1}},
\]
i.e.\ our matrix $\tQz_{1\scriptstyle\frac \ell{n+1}}$ is --- in the terminology of
\cite{B_imrn}
--- just the Stokes factor associated to the singular direction $\theta_{1\scriptstyle\frac \ell{n+1}}$.
This is valid for all $\ell\in\bZ$ because the formal monodromy is trivial in our situation; the matrices $\tQz_{1\scriptstyle\frac \ell{n+1}}$ are unchanged under $\ell\mapsto \ell+2n+2$.

It is now clear that Lemma \ref{lm:shapeofQeven} gives the same information as Lemma \ref{lm:shapeofQevenBOALCH}.  According to the former, the off-diagonal $(i,j)$ entry of $\tQz_{1\scriptstyle\frac \ell{n+1}}$ may be nonzero only when
\[
\arg(\omega^i-\omega^j) \equiv
\tfrac{n-\ell}{n+1}\pi \text{ mod } 2\pi;
\]
according to the latter, the condition is $\arg\ \alpha_{i,j}(-d_{n+1})=\theta_{1\scriptstyle\frac \ell{n+1}}$, that is
\[
\arg(\omega^j-\omega^i) \equiv
-\tfrac{\ell+1}{n+1}\pi \text{ mod } 2\pi,
\]
which is the same.  In particular, this information gives the roots supported by each singular direction. Summarizing:

\begin{pro}\label{pro:stokesfactoreven}
The matrix $\tQz_k$ is the Stokes factor corresponding to the singular direction $\theta_k$.  The roots $\cR(\theta_k)$
for $k=
1,
1\frac 1{n+1},
1\frac n{n+1}
$
are listed in Table \ref{t1even} below.  Here $n+1=2m$ and $m=2c$ or $2c+1$.
\end{pro}

\begin{table}[h]
\renewcommand{\arraystretch}{1.3}
\begin{tabular}{c|c|l}
$\theta$ & $m$ & $\{ (i,j) \ \vert\ \alpha_{i,j}\in\cR(\theta)\}$
\\
\hline
$\theta_1$ & $2c$ &
$\scriptstyle(2c-1,0),(2c-2,1),\dots,(c,c-1)$
\
$\scriptstyle(2c,4c-1),(2c+1,4c-2),\dots,(3c-1,3c)$
\\
$\theta_1$ & $2c+1$ &
$\scriptstyle(2c+1,4c+1),(2c+2,4c),\dots,(3c,3c+2)$
\
$\scriptstyle(2c,0),(2c-1,1),\dots,(c+1,c-1)$
\\
\hline
$\theta_{1\scriptstyle\frac 1{n+1}}$ & $2c$ &
$\scriptstyle(2c-1,4c-1),(2c,4c-2),\dots,(3c-2,3c)$
\
$\scriptstyle(2c-2,0),(2c-3,1),\dots,(c,c-2)$
\\
\vphantom{$\dfrac12$}
$\theta_{1\scriptstyle\frac 1{n+1}}$ & $2c+1$ &
$\scriptstyle(2c-1,0),(2c-2,1),\dots,(c,c-1)$
\
$\scriptstyle(2c,4c+1),(2c+1,4c),\dots,(3c,3c+1)$
\\
\hline
$\theta_{1\scriptstyle\frac n{n+1}}$ & $2c$ &
$\scriptstyle(4c-1,2c+1),(4c-2,2c+2),\dots,(3c+1,3c-1)$
\
$\scriptstyle(0,2c),(1,2c-1),\dots,(c-1,c+1)$
\\
$\theta_{1\scriptstyle\frac n{n+1}}$ & $2c+1$ &
$\scriptstyle(0,2c+1),(1,2c),\dots,(c,c+1)$
\
$\scriptstyle(4c+1,2c+2),(4c,2c+3),\dots,(3c+2,3c+1)$
\end{tabular}
\bigskip
\caption{Roots associated to Stokes factors when $n+1=2m$}
\label{t1even}
\end{table}

It is pointed out in \cite{B_imrn} that the roots associated to any \lq\lq half-period\rq\rq\  of singular directions, for example,
$\theta_1,\theta_{1\scriptstyle\frac 1{n+1}},\dots,\theta_{1\scriptstyle\frac n{n+1}}$,
form a system of positive roots.  This is a general property, which holds whenever one has a pole of order $2$.
In our situation the roots associated to $\theta_1$ and $\theta_{1\scriptstyle\frac n{n+1}}$ (the head and tail of the half-period) are precisely the corresponding simple roots.  This property, which is a special feature of the particular connection $\hat\alpha$,
follows by inspection from Table \ref{t1even}.  In \cite{GH} we shall give a more conceptual proof of this.

\begin{pro}\label{pro:simplerootseven} When $n+1=2m$ and $m=2c$, the roots associated to $\theta_1$ and $\theta_{1\scriptstyle\frac n{n+1}}$ form a system of simple roots.  This system is specified by the order relation $\prec$ of $0,1,\dots,4c-1$ shown below, where $i\prec j$ means that $i$ is to the left of $j$ in the diagram.

\begin{center}
\begin{tabular}{cccccccccccccc}
$3c$ & &$3c\!+\!1$ &$\ \dots\ $ & &$4c\!-\!1$ & & $0$& &$1$ &$\ \dots\ $ & &$c\!-\!1$ &
\\
 & $3c\!-\!1$ & &$\ \dots\ $ &$2c\!+\!1$ & &$2c$ & &$2c\!-\!1$ & &$\ \dots\ $ &$c\!+\!1$ & & $c$
\end{tabular}
\end{center}

\noindent Thus $\alpha_{a,b}$ is a positive root if and only if $a\succ b$, and $\alpha_{a,b}$ is a simple root if and only if $a,b$ occur consecutively in the form
$
\begin{smallmatrix}
b &  \\  & a
\end{smallmatrix}
$
or
$
\begin{smallmatrix}
& a \\ b &
\end{smallmatrix}
$
in the above diagram.
The form
$
\begin{smallmatrix}
b &  \\  & a
\end{smallmatrix}
$
indicates a root in
$\cR(\theta_1)$, and the form
$
\begin{smallmatrix}
& a \\ b &
\end{smallmatrix}
$
indicates a root in
$\cR(\theta_{1\scriptstyle\frac n{n+1}})$.

Similarly, when $n+1=2m$ and $m=2c+1$, the roots associated to $\theta_1$ and $\theta_{1\scriptstyle\frac n{n+1}}$ are the set of simple roots for the ordering of $0,1,\dots,4c+1$ shown below.

\begin{center}
\begin{tabular}{cccccccccccccc}
 & $3c\!+\!2$ & &$\ \dots\ $ & &$4c\!+\!1$ & & $0$& &$1$ &$\ \dots\ $& $c\!-\!1$ & & $c$
\\
$3c\!+\!1$ & &$3c$ & $\ \dots\ $ &$2c\!+\!2$ & &$2c\!+\!1$ & &$2c$ & &$\ \dots\ $ & & $c\!+\!1$ &
\end{tabular}
\end{center}

\noindent In this case also, the form
$
\begin{smallmatrix}
b &  \\  & a
\end{smallmatrix}
$
indicates a root in
$\cR(\theta_1)$, and the form
$
\begin{smallmatrix}
& a \\ b &
\end{smallmatrix}
$
indicates a root in
$\cR(\theta_{1\scriptstyle\frac n{n+1}})$.
\end{pro}

Thus our monodromy matrix $\tMz$ (Definition \ref{def:qqpi}) belongs to the subspace $\cM_{n+1}$
of $SL_{n+1}\bC$ defined as follows:

\begin{df}\label{def:meven}
When $n+1=2m$,
$\cM_{n+1}=
\{
(\underset{\alpha}\Pi \,U_\alpha)
(\underset{\beta}\Pi \,U_\beta)
\hat\Pi \ \vert\ \alpha\in\cR(\theta_1),\
\beta\in\cR(\theta_{1\scriptstyle\frac 1{n+1}})
\}.
$
\end{df}

It is clear from the explicit description in Proposition \ref{pro:simplerootseven} that
$E_{\alpha_{i,j}} E_{\alpha_{k,l}} = E_{\alpha_{k,l}} E_{\alpha_{i,j}} $
for any $\alpha_{i,j},\alpha_{k,l}\in \cR(\theta_1)$, because we have
either $(i,j)=(k,l)$ or $\{i,j\}\cap\{k,l\}=\emptyset$.
The same applies to $\cR(\theta_{1\scriptstyle\frac 1{n+1}})$.
Thus we have a stronger property than the fact (see Lemma \ref{lm:shapeofQevenBOALCH}) that the subgroups
$
\{
\underset{\alpha}\Pi \,U_\alpha
\ \vert\ \alpha\in\cR(\theta_1)
\}
$,
$
\{
\underset{\alpha}\Pi \,U_\beta
\ \vert\ \beta\in\cR(\theta_1)
\}
$
are independent of the ordering of the factors; it is actually true that
the individual factors commute in each subgroup.

\subsection{Monodromy data for the case $n+1=2m+1$}
$  $

This generalizes the case $n+1=5$ in section 4 of \cite{GIL1} and
Appendix B of \cite{GIL2}.
We shall just point out the differences from the case $n+1=2m$.

Diagonalization of the leading term $-W = \tilde P_0 \,(-d_{n+1}) \, \tilde P_0^{-1}$ at $\zeta=0$ involves a different choice of $\tilde P_0$, namely
$\tilde P_0=e^{-w}\, \Omega\,
\diag(1,\omega^{m+1},\omega^{2(m+1)},\dots,\omega^{n(m+1)})$.
Corresponding to this, at $\zeta=\infty$ we use
$\tilde P_\infty=e^{w}\, \Omega^{-1}\,
\diag(1,\omega^{m+1},\omega^{2(m+1)},\dots,\omega^{n(m+1)})^{-1}$.

Initial Stokes sectors at zero and infinity are taken to be
\[
\Omz_1=\{ \zeta\in\bC^\ast \ \vert \ -\tfrac{\pi}{2}-\tfrac{\pi}{2(n+1)}<\text{arg}\zeta <\tfrac\pi2+\tfrac{\pi}{2(n+1)} \} \ = \ \Omi_1,
\]
then we define
\[
\Omz_{k\frac{1}{n+1}}= e^{-\frac\pi{n+1}\i} \Omz_k,
\quad
\Omi_{k\frac1{n+1}}= e^{\frac\pi{n+1}\i} \Omi_k
\]
for any $k\in\tfrac1{n+1}\bZ$.

The symmetries of $\hat\alpha$ lead to the following symmetries of the holomorphic solutions
$\tPsiz_k$:
\begin{itemize}
\item
$d_{n+1}^{-1}\tPsiz_{k-\scriptstyle\frac2{n+1}}(\omega\zeta)\Pi^{-1}\omega^{m}=\tPsiz_k(\zeta),
$
\item
$(n+1)\Delta \left( \tPsiz_{k-1}(e^{\i\pi}\zeta)^{T}\right)^{-1}=\tPsiz_k(\zeta)$
\item
$(n+1)\Delta \tPsii_{k}(\frac 1{x^2\zeta}) = \tPsiz_{k}(\zeta)$
\item
$\overline{  \tPsiz_{\scriptstyle2-k}(\bar\zeta)  } C = \tPsiz_k(\zeta),
\
C=
\left(
\begin{smallmatrix}
1 & & & &\\
 & & & & \,1\\
 & & & \iddots &\\
 & 1 & & &
\end{smallmatrix}
\right)$.
\end{itemize}

From these symmetries we obtain:
\begin{itemize}
\item
$\tQz_{k\scriptstyle\frac2{n+1}} =
\Pi\, \tQz_k \, \Pi^{-1}
$
\item
$\tQz_{k+1} = \left( \tQz_k{}^{T}\right)^{-1}$
$\vphantom{\dfrac12}$
\item
$\tQz_{k}=\tQi_{k}$
$\vphantom{\dfrac12}$
\item
$\tQz_k=  C
\left(
\overline{\tQz_{ {\scriptstyle\frac{2n+1}{n+1} - k}  }}
\right)^{-1}  C$.
\end{itemize}

The analogue of Lemma \ref{lm:shapeofQeven} is:

\begin{lm}\label{lm:shapeofQodd}
When $n+1=2m+1$, the $(i,j)$ entry of the matrix $\tQz_{1\scriptstyle\frac\ell{n+1}}$ satisfies
\[
\left(
\tQz_{1\scriptstyle\frac\ell{n+1}}
\right)_{i,j}=
\begin{cases}
1 \quad \text{if}\ i=j
\\
0 \quad
\text{if $i\ne j$, and
$\arg(\omega^i-\omega^j) \not\equiv
(2n+1-2\ell)\tfrac{\pi}{2(n+1)}$ mod $2\pi$
}
\end{cases}
\]
\end{lm}

The monodromy of $\tPsiz_1$ is in this case
$
\tSz_1\tSz_2=
(\tQz_1\tQz_{1\scriptstyle\frac1{n+1}}\Pi)^{n+1},
$
so the analogue of Definition \ref{def:qqpi} is:

\begin{df}\label{def:qqpiodd}
When $n+1=2m+1$, $\tMz
\overset{\scriptstyle \text{\em def} }=
\tQz_1\tQz_{1\scriptstyle\frac1{n+1}}\Pi$.
\end{df}

Lemma \ref{lm:shapeofQevenBOALCH} still applies, but the singular directions are different in this case.  We still have $X=\diag(-1,-\omega,\dots,-\omega^n)$, but the arguments of the singular directions are $\tfrac\pi{2(n+1)}+\tfrac\pi{n+1}\bZ$ mod $2\pi$.
We choose initial singular direction
$\theta_1 = -\tfrac\pi{2(n+1)}$,
and denote further singular directions by
\[
\theta_{1\scriptstyle\frac \ell{n+1}} = -\tfrac{2\ell+1}{2(n+1)}\pi,\quad \ell\in\bZ.
\]

\begin{pro}\label{pro:stokesfactorodd}
The matrix $\tQz_k$ is the Stokes factor corresponding to the singular direction $\theta_k$.  The roots $\cR(\theta_k)$
for $k=
1,
1\frac 1{n+1},
1\frac n{n+1}
$
are listed in Table \ref{t1odd} below.  Here $n+1=2m+1$ and $m=2c$ or $2c+1$.
\end{pro}

\begin{table}[h]
\renewcommand{\arraystretch}{1.3}
\begin{tabular}{c|c|l}
$\theta$ & $m$ & $\{ (i,j) \ \vert\ \alpha_{i,j}\in\cR(\theta)\}$
\\
\hline
$\theta_1$ & $2c$ &
$\scriptstyle(2c,0),(2c-1,1),\dots,(c+1,c-1)$
\
$\scriptstyle(2c+1,4c),(2c+2,4c-1),\dots,(3c,3c+1)$
\\
$\theta_1$ & $2c+1$ &
$\scriptstyle(2c+1,0),(2c,1),\dots,(c+1,c)$
\
$\scriptstyle(2c+2,4c+2),(2c+3,4c+1),\dots,(3c+1,3c+3)$
\\
\hline
$\theta_{1\scriptstyle\frac 1{n+1}}$ & $2c$ &
$\scriptstyle(2c,4c),(2c+1,4c-1),\dots,(3c-1,3c+1)$
\
$\scriptstyle(2c-1,0),(2c-2,1),\dots,(c,c-1)$
\\
\vphantom{$\dfrac12$}
$\theta_{1\scriptstyle\frac 1{n+1}}$ & $2c+1$ &
$\scriptstyle(2c,0),(2c-1,1),\dots,(c+1,c-1)$
\
$\scriptstyle(2c+1,4c+2),(2c+2,4c+1),\dots,(3c+1,3c+2)$
\\
\hline
$\theta_{1\scriptstyle\frac n{n+1}}$ & $2c$ &
$\scriptstyle(4c,2c+2),(4c-1,2c+3),\dots,(3c+2,3c)$
\
$\scriptstyle(0,2c+1),(1,2c),\dots,(c,c+1)$
\\
$\theta_{1\scriptstyle\frac n{n+1}}$ & $2c+1$ &
$\scriptstyle(0,2c+2),(1,2c+1),\dots,(c,c+2)$
\
$\scriptstyle(4c+2,2c+3),(4c+1,2c+4),\dots,(3c+3,3c+2)$
\end{tabular}
\bigskip
\caption{Roots associated to Stokes factors when $n+1=2m+1$}
\label{t1odd}
\end{table}

From this we obtain the analogue of Proposition \ref{pro:simplerootseven}:

\begin{pro}\label{pro:simplerootsodd} When $n+1=2m+1$ and $m=2c$, the roots associated to $\theta_1$ and $\theta_{1\scriptstyle\frac n{n+1}}$ form a system of simple roots.  This system is specified by
the order relation $\prec$ of $0,1,\dots,4c$ shown below, where $i\prec j$ means that $i$ is to the left of $j$ in the diagram.

\begin{center}
\begin{tabular}{cccccccccccccc}
$3c+1$ & &$3c\!+\!2$ &$\ \dots\ $ & &$4c$ & & $0$& &$1$ &$\ \dots\ $ &$c\!-\!1$ & &c
\\
 & $3c$ & &$\ \dots\ $ &$2c\!+\!2$ & &$2c\!+\!1$ & &$2c$ & &$\ \dots\ $ & &$c\!+\!1$ &
\end{tabular}
\end{center}

\noindent Thus $\alpha_{a,b}$ is a positive root if and only if $a\succ b$,
and $\alpha_{a,b}$ is a simple root if and only if $a,b$ occur consecutively in the form
$
\begin{smallmatrix}
b &  \\  & a
\end{smallmatrix}
$
or
$
\begin{smallmatrix}
& a \\ b &
\end{smallmatrix}
$
in the above diagram.
The form
$
\begin{smallmatrix}
b &  \\  & a
\end{smallmatrix}
$
indicates a root in
$\cR(\theta_1)$, and the form
$
\begin{smallmatrix}
& a \\ b &
\end{smallmatrix}
$
indicates a root in
$\cR(\theta_{1\scriptstyle\frac n{n+1}})$.

Similarly, when $n+1=2m+1$ and $m=2c+1$, the roots associated to $\theta_1$ and $\theta_{1\scriptstyle\frac n{n+1}}$ are the set of simple roots for the ordering of $0,1,\dots,4c+2$ shown below.

\begin{center}
\begin{tabular}{cccccccccccccc}
 & $3c\!+\!3$ & &$\ \dots\ $ & &$4c\!+\!2$ & & $0$& &$1$ &$\ \dots\ $& & $c$ &
\\
$3c\!+\!2$ & &$3c\!+\!1$ & $\ \dots\ $ &$2c\!+\!3$ & & $2c\!+\!2$ & &$2c\!+\!1$ & &$\ \dots\ $ & $c\!+\!2$&  & $c\!+\!1$
\end{tabular}
\end{center}

\noindent In this case also, the form
$
\begin{smallmatrix}
b &  \\  & a
\end{smallmatrix}
$
indicates a root in
$\cR(\theta_1)$, and the form
$
\begin{smallmatrix}
& a \\ b &
\end{smallmatrix}
$
indicates a root in
$\cR(\theta_{1\scriptstyle\frac n{n+1}})$.
\end{pro}

The analogue of Definition \ref{def:meven} is:

\begin{df}\label{def:modd}
When $n+1=2m+1$,
$\cM_{n+1}=
\{
(\underset{\alpha}\Pi \,U_\alpha)
(\underset{\beta}\Pi \,U_\beta)
\Pi \ \vert\ \alpha\in\cR(\theta_1),\
\beta\in\cR(\theta_{1\scriptstyle\frac 1{n+1}})
\}.
$
\end{df}

\section{Conjugacy classes}\label{steinberg}

We begin by quoting some results from \cite{S}.
Let $G$ be a linear reductive algebraic group over $\bC$ and $n$ its rank (the dimension of an algebraic maximal torus $T$).  Linearity means that $G$ may be regarded as a subgroup of $GL(V)$ for some finite-dimensional complex vector space $V$.

Recall that an element of $G$ is said to be regular if its stabilizer has minimal dimension (necessarily $n$), and semisimple if its adjoint representation is semisimple. We denote by $G^{reg}$ the set of all regular elements of $G$, and by $G^{ss}$ the set of all semisimple elements.  For any subset $X\subseteq G$ which is invariant under conjugation by elements of $G$, we denote by $X/G$ the set of conjugacy classes.
Since $G$ is not compact, $G/G$ and in general $X/G$ are not Hausdorff. To obtain a Hausdorff quotient, one should take $G^{ss}/G$ and in general $X^{ss}/G$, as semisimple elements here are precisely the elements whose orbits under conjugation are closed (cf.\ \cite{Ri}). This type of quotient is sometimes called the categorical quotient.

Recall also that, for an endomorphism $g$ of $V$, the multiplicative Jordan-Chevalley decomposition $g=g^{ss}g^{u}=g^{u}g^{ss}$ exists, and the factors are unique, where
$g^{ss}$ is semisimple and $g^{u}$ is unipotent.  Furthermore, if $g\in G$, then also $g^{ss},g^{u}\in G$.

\begin{pro}\cite{S}(Theorem 3, p.115)\label{pro:S1}
There is a bijection $G^{reg}/G\cong G^{ss}/G$ given by $g\mapsto g^{ss}$, i.e.\ by sending an element to its semi-simple part.
\end{pro}

Consider the map
$
\chi: G \to \bC^n, \quad g\mapsto (\chi_1,\dots,\chi_n)
$
where the $\chi_i$ are the characters of the basic irreducible representations.  Clearly each fibre of this map is a union of conjugacy classes.

\begin{pro}\cite{S}(Theorem 4, p.120)\label{pro:S2}
Assume further that $G$ is simply connected and semi-simple.
The map
\[
s:\bC^n\to G^{reg},\quad
(t_1,\dots,t_n)  \mapsto e_1(t_1)\n_1\dots e_n(t_n)\n_n
\]
is a cross-section of $\chi|_{G^{reg}}:G^{reg}\to \bC^n$, where
$e_i(t_i)=\exp(t_i E_{\alpha_i})$, $\simple=\{\alpha_1,\dots,\alpha_n\}$ is an ordered set of simple roots,  the $E_{\alpha_i}$ are corresponding root vectors, and the $\n_i$ are representatives in $N(T)$ of the corresponding elements in the Weyl group $N(T)/T$.
\end{pro}

Thus $\bC^n$  parametrizes the {\em regular} conjugacy classes in $G$, and the cross-section gives a particular choice of representatives.
In the case $G=SL_{n+1}\bC$, for the standard choice of simple roots
$\alpha_{1,0},\alpha_{2,1},\dots,\alpha_{n,n-1}$,
this choice of representatives is nothing but the rational normal form of a cyclic matrix (a single companion matrix).  However, in order to deal with other choices (which arise from the Stokes analysis), we make the following definition.

\begin{df}\label{def:rpi}  Let $G=SL_{n+1}\bC$.  Using the notation of Proposition \ref{pro:S2}, we define
\[
\cR_{n+1}^\simple=\{  e_1(t_1)\n_1\dots e_n(t_n)\n_n  \ \vert\ (t_1,\dots,t_n)\in \bC^n \}
\quad(\subseteq SL_{n+1}\bC)
\]
where $\simple$ is a given ordered set of simple roots.
\end{df}

The main result of this section, whose proof will be given in the next two subsections,  is:

\begin{thm}\label{thm:m-is-r}
$\cM_{n+1}=\cR_{n+1}^\simple$, where $\simple =
\cR(\theta_1)\cup\cR(\theta_{1\scriptstyle\frac n{n+1}})$.
In other words, the space $\cM_{n+1}$ in the Stokes analysis of section \ref{boalch} is actually (the image of) a Steinberg cross-section of the space of regular conjugacy classes. This holds for any ordering of $\simple$ in which $\cR(\theta_1)$ precedes $\cR(\theta_{1\scriptstyle\frac n{n+1}})$.
\end{thm}

This will give useful information on the structure of the matrices
$\tMz$ (and, more generally, any
matrices in $\cM_{n+1}$). In particular, we have:

\begin{cor}\label{cor:mstructure}
(i) The entries of the matrix $\tMz$ are given by explicit (polynomial) functions of the coefficients of its characteristic polynomial. (ii) $\tMz$ is real if and only if its characteristic polynomial is real.
\end{cor}

\begin{proof} (i) We have $\tMz= s(\chi(\tMz))$, where
$s(t_1,\dots,t_n)= e_1(t_1)\n_1\dots e_n(t_n)\n_n$ as in Proposition \ref{pro:S2}.
The map $\chi:SL_{n+1}\bC \to \bC^n$ gives the coefficients of the characteristic polynomial.
(ii) This follows from (i) and the fact that all $e_i,\n_i$ may be chosen to be real.
\end{proof}

\subsection{Proof of Theorem \ref{thm:m-is-r} in the case $n+1=2m$}
$  $

\begin{proof}  Recall that
$\cM_{n+1}=
\{
(\underset{\alpha}\Pi\,U_\alpha)(\underset{\beta}\Pi\,U_\beta)
\hat\Pi \ \vert\ \alpha\in\cR(\theta_1),\
\beta\in\cR(\theta_{1\scriptstyle\frac 1{n+1}})
\}
$
(Definition \ref{def:meven}). We shall produce this set from
\[
\cR_{n+1}^\simple=
\textstyle\prod_{\alpha\in \cR(\theta_1)\cup\cR(\theta_{1\frac n{n+1}})}
e_\alpha \n_\alpha
\]
by \lq\lq moving all the $\n_\alpha$ to the right\rq\rq.  Here, for brevity, $e_\alpha$ denotes $e_\alpha(\bC)=U_\alpha$.

It is well known that $\n_{\alpha_{i,j}}$ can be chosen so that
\[
\n_{\alpha_{i,j}} \exp \bC E_{k,l} \ \n_{\alpha_{i,j}}^{-1} =
\exp \bC E_{p(k),p(l)}
\]
where $p$ is the permutation which interchanges $i$ and $j$.

First we consider $\cR(\theta_1)$.  From Proposition \ref{pro:simplerootseven}  we have
\[
\textstyle
\prod_{\alpha\in \cR(\theta_1)} e_\alpha \n_\alpha =
\prod_{\alpha\in \cR(\theta_1)} e_\alpha
\prod_{\alpha\in \cR(\theta_1)} \n_\alpha
\]
as, for any
$
\begin{smallmatrix}
b_1 &  \\  & a_1
\end{smallmatrix}
$,
$
\begin{smallmatrix}
b_2 &  \\  & a_2
\end{smallmatrix}
$
(representing elements of $\cR(\theta_1)$) either
$(a_1,b_1)=(a_2,b_2)$ or $\{a_1,b_1\}\cap\{a_2,b_2\}=\emptyset$
(cf.\ the remark after Definition \ref{def:meven}).
Hence $\n_{\alpha_1} e_{\alpha_2}  = e_{\alpha_2}  \n_{\alpha_1}$ for any
$\alpha_1,\alpha_2\in\cR(\theta_1)$.
The same argument applies to
$\cR(\theta_{1\scriptstyle\frac n{n+1}})$, so
\[
\textstyle
\prod_{\alpha\in \cR(\theta_{1\frac n{n+1}})} e_\alpha \n_\alpha =
\prod_{\alpha\in \cR(\theta_{1\frac n{n+1}})} e_\alpha
\prod_{\alpha\in \cR(\theta_{1\frac n{n+1}})} \n_\alpha
\]
as well.
Next we examine
\[
\textstyle
\n_\ast \,e_\beta = \n_\ast e_\beta\,  \n_\ast^{-1} \  \n_\ast, \quad
\n_\ast= \prod_{\alpha\in \cR(\theta_1)} \n_\alpha
\]
for each $\beta\in \cR(\theta_{1\scriptstyle\frac n{n+1}})$.

The effect of conjugation by $\n_\ast$ is to perform all reflections
$
\begin{smallmatrix}
b &  \\  & a
\end{smallmatrix}
\mapsto
\begin{smallmatrix}
a &  \\  & b
\end{smallmatrix}.
$
In the case $m=2c$,
this converts the first diagram in Proposition \ref{pro:simplerootseven} to

\begin{center}
\begin{tabular}{cccccccccccccc}
$3c\!-\!1$ & &$3c\!-\!2$ &$\ \dots\ $ & &$2c$ & & $2c\!-\!1$& &$2c\!-\!2$ &$\ \dots\ $ & &$c$ &
\\
 & $3c$ & &$\ \dots\ $ &$4c\!-\!2$ & &$4c\!-\!1$ & &$0$ & &$\ \dots\ $ &$c\!-\!2$ & & $c\!-\!1$
\end{tabular}
\end{center}

This means that the image of the roots $\cR(\theta_{1\scriptstyle\frac n{n+1}})$ (or root spaces) under conjugation by $\n_\ast$ is given by those $\alpha_{i,j}$ with
\[
(i,j)=(3c-2,3c),\dots,(2c,4c-2),(2c-1,4c-1),(2c-2,0),\dots,(c,c-2)
\]
and this is precisely $\cR(\theta_{1\scriptstyle\frac 1{n+1}})$.
In the case $m=2c+1$,  a similar argument applies to the second
diagram in Proposition \ref{pro:simplerootseven}.

We conclude that
\[
\textstyle
\prod_{\alpha\in \cR(\theta_1)\cup\cR(\theta_{1\frac n{n+1}})} e_\alpha \n_\alpha
=
\prod_{\alpha\in \cR(\theta_1)} e_\alpha
\ \
\prod_{\alpha\in \cR(\theta_{1\frac 1{n+1}})} e_\alpha
\ \
\prod_{\alpha\in \cR(\theta_1)\cup\cR(\theta_{1\frac n{n+1}})}\n_\alpha.
\]
The effect of conjugation by $\prod_{\alpha\in \cR(\theta_1)\cup\cR(\theta_{1\frac n{n+1}})}\n_\alpha$ is the cyclic permutation $(0\,1\,\cdots\,n)$, which may be represented by the matrix $\hat\Pi$.  Hence
$\prod_{\alpha\in \cR(\theta_1)\cup\cR(\theta_{1\frac n{n+1}})} e_\alpha \n_\alpha=\cM_{n+1}$.
\end{proof}

\subsection{Proof of Theorem \ref{thm:m-is-r} in the case $n+1=2m+1$}
$  $

This is exactly the same as the proof in the case $n+1=2m$, but using Proposition \ref{pro:simplerootsodd} instead of Proposition \ref{pro:simplerootseven}.  The key step is to see the
effect of conjugation by $\n_\ast$, i.e.\ of applying all reflections
$
\begin{smallmatrix}
b &  \\  & a
\end{smallmatrix}
\mapsto
\begin{smallmatrix}
a &  \\  & b
\end{smallmatrix}
$
to the diagrams in Proposition \ref{pro:simplerootsodd}. In the case $m=2c$, with the first diagram, we obtain

\begin{center}
\begin{tabular}{cccccccccccccc}
$3c$ & &$3c\!-\!1$ &$\ \dots\ $ & &$2c\!+\!1$ & & $2c$& &$2c\!-\!1$ &$\ \dots\ $ & $c\!+\!1$& &$c$
\\
 & $3c\!+\!1$ & &$\ \dots\ $ &$4c\!-\!1$ & &$4c$ & &$0$ & &$\ \dots\ $ & & $c\!-\!1$&
\end{tabular}
\end{center}

This means that the image of the roots $\cR(\theta_{1\scriptstyle\frac n{n+1}})$  under conjugation by $\n_\ast$ is given by those $\alpha_{i,j}$ with
\[
(i,j)=(3c-1,3c+1),\dots,(2c+1,4c-1),(2c,4c),(2c-1,0),\dots,(c,c-1)
\]
and this is precisely $\cR(\theta_{1\scriptstyle\frac 1{n+1}})$. The case $m=2c+1$,  with the second diagram, is similar. We conclude again that
$\prod_{\alpha\in \cR(\theta_1)\cup\cR(\theta_{1\frac n{n+1}})} e_\alpha \n_\alpha=\cM_{n+1}$.

\subsection{A lemma on orbits of the maximal compact subgroup}
$  $

We shall need the following lemma, which may be well known, but for lack of a suitable reference we give the proof.

\begin{lm}\label{lm:AdGK}
Let $K$ be a maximal compact subgroup of $G$. Let $\fg=\fk+\fp$ be the Cartan decomposition of the Lie algebra $\fg$ of $G$, where $\fk$ is the Lie algebra of $K$.
Denote by $\cC_k^G$ the $G$-orbit of an element $k\in K$ under conjugation,
and by $\cC_k^K$ the $K$-orbit.
Then
$\cC_k^G\cap K=\cC_k^K$.
\end{lm}

\begin{proof}
Recall first that any element $g$ in $G$ can be written as $g=k\exp X$ where $k\in K$ and $X\in \fp$, since the multiplication map $K\times \exp \fp \to G$ is a diffeomorphism
(p.\  384 of \cite{Kn}).  To prove that $\cC_k^G\cap K=\cC_k^K$,  it suffices to show that if $gk_1g^{-1}=k_2$ for $k_1,k_2\in K$ and $g\in G$, then there exists $h\in K$ such that $hk_1h^{-1}=k_2$.

Writing $g=k\exp X$ for some $k\in K$ and $X\in\fp$, we have $\exp X k_1 \exp (-X)= k^{-1} k_2 k \in K$. Let $\Theta$ be the global Cartan involution on the group (p.\ 387 of \cite{Kn}). Then $\Theta(\exp(X)k_1\exp(-X))=\exp(-X)k_1\exp(X)$.
As this lies in $K$, it is fixed by $\Theta$, thus $\exp(X) k_1 \exp(-X)=\exp(-X)k_1\exp(X)$. In other words, $\exp(-2X)k_1\exp(2X)=k_1$, and so $k_1 \exp(2X) k_1^{-1}=\exp (2X)\in\exp \fp$.

The restriction of the diffeomorphism $K\times\exp \fp\to G$ is a diffeomorphism from $\fp$ onto $\exp\fp$.
Thus, $\Ad_{k_1}(2X)\in \fp$  since $\exp (\Ad_{k_1} (2X))=k_1\exp(2X) k_1^{-1}\in\exp\fp$.  Moreover, $2X=\Ad_{k_1}(2X)$ since exp is injective on $\fp$. This
implies $\exp(\Ad_{k_1}X)=\exp X$ and thus $k_2=k(\exp X k_1 \exp (-X))k^{-1}=k(k_1)k^{-1}$ where $k\in K$.
\end{proof}

This implies that the set of $G$-conjugacy classes of
$\{ gkg^{-1} \ \vert\ g\in G,k\in K\}$ can be identified with the set of $K$-conjugacy classes of $K$.  It is well known that the latter can be identified with a Weyl alcove of $K$.

\section{Solutions of the  tt$^*$-Toda equations}\label{localsolutions}

From now on we shall consider those matrices $\tMz$ which arise from solutions $w_i$ (of the $tt^*$-Toda equations) defined on intervals of the form $(0,\epsilon)$, such that $w_i=w_i(|t|)$ and $2w_i\sim \gamma_i\log |t|$ as $t\rightarrow 0$.
(Thus the global solutions are the solutions of this type with $\epsilon=\infty$.)
We refer to such solutions as {\em smooth near zero}.

\begin{pro}\label{pro:characterizationbygamma}  (i) For any
solution $w_i$ defined on an interval of the form $(0,\epsilon)$, such that $w_i=w_i(|t|)$ and $2w_i\sim \gamma_i\log |t|$ as $z\rightarrow 0$, we have
$\gamma_{i+1}-\gamma_i\geq -2$ for all $i$. (ii) Conversely, for any $(\gamma_0,\dots,\gamma_n)\in\bR^{n+1}$ such that $\gamma_{i+1}-\gamma_i\geq -2$ for all $i$, there exists a solution $w_i$ such that $w_i$ arises as in (i).
\end{pro}

\begin{proof}   (i) is elementary \cite{GL}, and (ii) follows from
Theorem \ref{thm:convex}.  Alternatively, without appealing to the result on global solutions, (ii) follows from the Iwasawa factorization method of \cite{GIL3}, which is purely local.
\end{proof}

\begin{df}\label{def:m0}
We denote by $\cM_{n+1}^{\text{\em zero}}$ the subset of $\cM_{n+1}$ consisting of matrices $\tMz$ arising from solutions $w_i$ of the $tt^*$-Toda equations which are smooth near zero.
\end{df}

For such solutions, we have the following result, which will be proved in Section \ref{app} (an equivalent result can be found in \cite{M1},\cite{M2}).

\begin{thm} \label{thm:eigenvalues}  If $\tMz\in \cM_{n+1}^{\text{\em zero}}$, then its eigenvalues are in the unit circle.  More precisely, if $n+1=2m$, then the eigenvalues of
$\tMz$ are
\[
\text{
$e^{\pm\frac{\pi \i}{n+1}(\gamma_0 +1)}$,
$e^{\pm\frac{\pi\i}{n+1}(\gamma_1 +3)}$,
$\dots$,
$e^{\pm\frac{\pi\i}{n+1}(\gamma_{m-1} +  2m-1)}$
}
\]
and, if $n+1=2m+1$, they are
\[
\text{
$e^{\pm\frac{\pi \i}{n+1}(\gamma_0 -n)}$,
$e^{\pm\frac{\pi\i}{n+1}(\gamma_1 -n+2)}$,
$\dots$,
$e^{\pm\frac{\pi\i}{n+1}(\gamma_{m-2} -4)}$,
$e^{\pm\frac{\pi\i}{n+1}(\gamma_{m-1} -2)}$,
$1$.
}
\]
\end{thm}

From section \ref{steinberg}
we know that $\cM_{n+1}^{\text{zero}}$ is contained in
$(\cM_{n+1}=)\ \cR_{n+1}^\simple$.
From Theorem \ref{thm:eigenvalues} we see that $\cM_{n+1}^{\text{zero}}$ is contained in the smaller subspace
\[
\cR_{n+1}^{\simple, c,p} = \{ X\in \cR_{n+1}^{\simple} \ \vert\ X \text{ satisfies (c) and (p) }\}
\]
where the conditions on $X$ are

(c) the eigenvalues of $X$ are in the unit circle

(p) if $x$ is an eigenvalue of $X$, then so is $1/x$.

\noindent We regard condition (c) as a compactness or \lq\lq unitary\rq\rq\ condition (on the eigenvalues).  Condition (p) means that the characteristic polynomial of $X$ is palindromic in the case $n+1=2m$ (anti-palindromic in the case $n+1=2m+1$).

\begin{cor}\label{cor:real}
All matrices $\tMz\in\cM_{n+1}^{\text{\em zero}}$ (hence all matrices $\tQz_k$)
are real.
\end{cor}

\begin{proof}  Theorem \ref{thm:eigenvalues} shows that the roots of the characteristic polynomial of $\tMz$ are real or come in complex conjugate pairs, hence the coefficients of this polynomial are real.
By Corollary \ref{cor:mstructure}, $\tMz$ itself must be real.  It follows that
$\tQz_1\tQz_{1\scriptstyle\frac1{n+1}}$, hence the individual factors $\tQz_1$, $\tQz_{1\scriptstyle\frac1{n+1}}$, are also real.  All other $\tQz_k$ can be determined from $\tQz_1$, $\tQz_{1\scriptstyle\frac1{n+1}}$ by using the symmetries, and it is clear from these formulae that the $\tQz_k$  are real.
\end{proof}

We remark that Theorem \ref{thm:eigenvalues} (in conjunction with Corollary \ref{cor:mstructure})
also permits the explicit computation of the entries of the matrix $\tMz$ (and the matrices $\tQz_k$) in terms of the data $\gamma_0,\dots,\gamma_n$ of the local solutions.  Some examples will be given in the next section.

Let us look at the structure of $\cM_{n+1}^{\text{zero}}$ in more detail.
First we consider the space $\cR_{n+1}^{\simple, c}$.

\begin{lm}\label{lm:alcove} $\cR_{n+1}^{\simple, c}$ can be identified with (any) Weyl alcove $\fA_{n+1}$ of the maximal compact subgroup $K=SU_{n+1}$ of $G=SL_{n+1}\bC$.
\end{lm}

\begin{proof}  The Steinberg maps in Proposition \ref{pro:S1} and \ref{pro:S2} both preserve the eigenvalues, so we may impose condition (c) on all three spaces.  Then we have
\begin{align*}
\cR_{n+1}^{\simple, c} & \cong G^{reg,c}/G \quad\text{by Proposition \ref{pro:S2} }
\\
& \cong  G^{ss,c}/G \quad \text{by Proposition \ref{pro:S1} }
\\
& = \cC^G_K/G
\end{align*}
where $ \cC^G_K=\{ gkg^{-1} \ \vert\ g\in G,k\in K\}$.
(It is clear that $ \cC^G_K\subseteq G^{ss,c}$, and the reverse inclusion holds because a semisimple element of $G=SL_{n+1}\bC$ is diagonalizable.)  Next we have
\begin{align*}
\cC^G_K/G & \cong K/K \quad \text{by Lemma \ref{lm:AdGK} }
\\
& \cong  \fA_K
\end{align*}
where $\fA_K$ is any Weyl alcove of $K=SU_{n+1}$.
\end{proof}

Now we can give a Lie-theoretic description of the monodromy data $\cM_{n+1}^{\text{zero}}$
for the solutions of the tt*-Toda equations that are smooth near zero.

\begin{thm}\label{thm:main}   The map
$\cM_{n+1}^{\text{\em zero}} \to \cR_{n+1}^{\simple,c,p}$ is bijective.  The space
$\cR_{n+1}^{\simple,c,p}$ can be identified in a natural way
with a convex subset $\fA_{n+1}^p$ of
a Weyl alcove $\fA_{n+1}$ of the maximal compact subgroup $SU_{n+1}$ of $SL_{n+1}\bC$.
\end{thm}

The convex subset $\fA_{n+1}^p$ is obtained by imposing condition (p) on $\fA_{n+1}$. It will be defined explicitly in the proof of the theorem.

\subsection{Proof of Theorem \ref{thm:main} in the case $n+1=2m$}
$  $

Let us choose the maximal torus
\[
T_{n+1}=\{
\diag(e^{\pi\i\varrho_0},\dots,e^{\pi\i\varrho_n})
\ \vert\
\varrho_i\in\bR,
\textstyle \sum_{i=0}^n \varrho_i=0
\}
\]
of $SU_{n+1}$, and the
Weyl alcove
\[
\fA_{n+1}=
\{(\varrho_0,\dots,\varrho_n)\in\bR^{n+1}
\ \vert \
\varrho_m\leq\dots\leq\varrho_{2m-1}\leq\varrho_0\leq \dots \leq\varrho_{m-1}\leq2+\varrho_m,\
\textstyle \sum_{i=0}^n \varrho_i=0
\}
\]
This makes explicit the identification $K/K \cong \fA_{n+1}$. Condition (p)
is now equivalent to $\varrho_i+\varrho_{n-i}=0$.  Imposing this, we obtain
$(K/K)^p \cong \fA_{n+1}^p$.

By Theorem \ref{thm:eigenvalues} and the proof of Lemma \ref{lm:alcove}, we see that the map
\[
\cM_{n+1}^{\text{zero}} \to \cR_{n+1}^{\simple,c} \cong \fA_{n+1}
\]
is given explicitly by
\[
\tMz \mapsto
(\varrho_0,\dots,\varrho_n)=(\tfrac{\gamma_0+1}{n+1}, \tfrac{\gamma_1+3}{n+1},\dots,\tfrac{\gamma_{m-1}+2m-1}{n+1},
\tfrac{1-2m-\gamma_{m-1}}{n+1}, \dots,\tfrac{-1-\gamma_0}{n+1}).
\]
By Proposition \ref{pro:characterizationbygamma},
its image consists of the points of this form which also satisfy $\gamma_{i+1}-\gamma_i\ge -2$.
This is exactly the set $\cR_{n+1}^{\simple,c,p} \cong \fA_{n+1}^p$.

Finally we point out that
the specific convex region which arises from the  tt*-Toda equations
can be expressed as

\noindent
$(\dagger)$\quad\quad\quad\quad
$\{(\gamma_0,\dots,\gamma_n)\in\bR^{n+1}
\ \vert\
 \gamma_{i+1}-\gamma_i\geq -2, \gamma_{n-i}+\gamma_i=0\}$
 \newline
 \quad\quad\quad\quad\quad\quad\quad\quad\makebox[4cm]{}
$=
(n+1)\fA_{n+1}^p -
(1,3,\dots,2m-1,1-2m,\dots,-3,-1)$

\noindent
i.e.\
a rescaling of the convex subset $\fA_{n+1}^p$ followed by
translation by
the sum of the (positive) coroots, where positive refers to the
alcove chosen above.

The rescaling by a factor of $n+1$ corresponds to using the full monodromy
\[
(\tMz\omega^{\frac12})^{n+1}
=
-(\tQz_1\tQz_{1\scriptstyle\frac1{n+1}}\hat\Pi)^{n+1}
\]
of the connection $d+\hat\alpha$, rather than our matrix $\tMz$.  However, it would not have been possible to work with the full monodromy throughout this article, as
$(\tMz\omega^{\frac12})^{n+1}$ is not in general regular,
so the methods of section \ref{steinberg} would not apply.

\subsection{Proof of Theorem \ref{thm:main} in the case $n+1=2m+1$}
$  $

This is very similar to the previous case.  We take the
Weyl alcove
\[
\fA_{n+1}=
\{(\varrho_0,\dots,\varrho_n)\in\bR^{n+1}
\ \vert \
\varrho_0\leq\varrho_1\leq \dots \leq\varrho_{2m}\leq 2+\rho_0,\
\textstyle \sum_{i=0}^n \varrho_i=0
\}.
\]
Condition (p)
is now equivalent to
$\varrho_i+\varrho_{2m-i}=0$ for all $i$.

The map
\[
\cM_{n+1}^{\text{zero}} \to \cR_{n+1}^{\simple,c} \cong \fA_{n+1}
\]
is given explicitly by
\[
\tMz \mapsto
(\varrho_0,\dots,\varrho_n)=(\tfrac{\gamma_0-n}{n+1}, \tfrac{\gamma_1-(n-2)}{n+1},\dots,\tfrac{\gamma_{m-1}-2}{n+1},0,
\tfrac{2-\gamma_{m-1}}{n+1}, \dots,\tfrac{n-\gamma_0}{n+1}).
\]
By Proposition \ref{pro:characterizationbygamma},
its image consists the points of this form which also satisfy $\gamma_{i+1}-\gamma_i\ge -2$.
This is exactly the set $\cR_{n+1}^{\simple,c,p} \cong \fA_{n+1}^p$.

The specific convex region which arises from the  tt*-Toda equations
can be expressed as

\noindent
$(\dagger\dagger)$\quad\quad\quad\quad
$\{(\gamma_0,\dots,\gamma_n)\in\bR^{n+1}
\ \vert\
 \gamma_{i+1}-\gamma_i\geq -2, \gamma_{n-i}+\gamma_i=0\}$
 \newline
\quad\quad\quad\quad\quad\quad\quad\quad\makebox[4cm]{}
$=
(n+1)\fA_{n+1}^p -
(-2m,-2m+2,\dots,-2,0,2,\dots,2m)
$

\noindent
i.e.\
a rescaling of the convex subset $\fA_{n+1}^p$ followed by
translation by the sum of the (positive) coroots, where positive refers to the
alcove chosen above.

\subsection{Conclusion: solutions of the tt*-Toda equations}
$  $

Let us denote by
$
\cS_{n+1}^{\text{global}}
\ \subseteq\
\cS_{n+1}^{\text{zero}}
\ \subseteq\
\cS_{n+1}^{\text{local}}
$
the spaces of solutions $(w_0,\dots,w_n)$ of the $tt^*$-Toda equations (\ref{ost}) which are smooth (respectively) on $\bC^\ast$, on some open set of the form
$\{t\ \vert\ 0<\vert t\vert < \epsilon\}$, or on some open set of the form
$\{t\ \vert\ a<\vert t\vert < b\}$ (where $a$ can be $0$ and $b$ can be $\infty$).

We also have the corresponding spaces
$
\cM_{n+1}^{\text{global}}
\ \subseteq\
\cM_{n+1}^{\text{zero}}
\ \subseteq\
\cM_{n+1}^{\text{local}}
$
of matrices $\tMz$, where the $n\!+\!1$-th power of $\tMz$ is, up to sign, the monodromy (see section \ref{boalch}). These spaces are related to the solution spaces as follows
\[
\begin{CD}
\cS_{n+1}^{\text{global}} @>{\subseteq}>>  \cS_{n+1}^{\text{zero}}  @>{\subseteq}>>  \cS_{n+1}^{\text{local}}
\\
@VVV   @VVV   @VVV
\\
\cM_{n+1}^{\text{global}} @>{\subseteq}>>  \cM_{n+1}^{\text{zero}}  @>{\subseteq}>>  \cM_{n+1}^{\text{local}}
\end{CD}
\]
where each vertical map is the \lq\lq monodromy map\rq\rq\ (and is surjective, by definition).

In this article we do not discuss the map $\cS_{n+1}^{\text{local}} \to \cM_{n+1}^{\text{local}}$.  The map $\cS_{n+1}^{\text{zero}} \to \cM_{n+1}^{\text{zero}}$ is surjective but not bijective; its fibres are given by the possible values of the connection matrix $\tilde E_1$ (which is computed in \cite{GIL3}).  The map $\cS_{n+1}^{\text{global}} \to \cM_{n+1}^{\text{global}}$ is our main interest.

Theorem \ref{thm:main} implies that
the composition
\begin{gather*}
\cS_{n+1}^{\text{zero}} \to \cM_{n+1}^{\text{zero}} \to \cR_{n+1}^{\simple,c}
\overset{\cong}\to
\fA_{n+1}
\\
(w_0,\dots,w_n)\mapsto \text{conjugacy class of semisimple part of $\tMz$}
\end{gather*}
is surjective onto the convex subset $\fA_{n+1}^p$
(because the map $\cM_{n+1}^{\text{zero}} \to \cR_{n+1}^{\simple,c,p}
\overset{\cong}\to \fA_{n+1}^p$ is bijective).

On the other hand, because of $(\dagger)$ and $(\dagger\dagger)$ above,
Theorem \ref{thm:convex} can be interpreted as saying that
the compositions
\[
\cS_{n+1}^{\text{global}}
\hookrightarrow
\cS_{n+1}^{\text{zero}}
\to \cM_{n+1}^{\text{zero}} \to
\cR_{n+1}^{\simple,c,p}
\overset{\cong}\to
\fA_{n+1}^p
\]
and
\[
\cS_{n+1}^{\text{global}} \to
\cM_{n+1}^{\text{global}} \to
\cR_{n+1}^{\simple,c,p}
\overset{\cong}\to
\fA_{n+1}^p
\]
are bijective.   It follows that the maps
\[
\cM_{n+1}^{\text{global}} \hookrightarrow \cM_{n+1}^{\text{zero}}
\]
and
\[
\cS_{n+1}^{\text{global}} \hookrightarrow \cM_{n+1}^{\text{global}}
\]
are both bijective.
In particular we obtain the following Lie-theoretic formulation of Theorem \ref{thm:convex}:

\begin{pro}\label{pro:lietheoretic}
The map
\begin{gather*}
\cS_{n+1}^{\text{\em global}} \overset{\cong}\to \cM_{n+1}^{\text{\em global}} \to \cR_{n+1}^{\simple,c,p}
\overset{\cong}\to \fA_{n+1}^p
\\
(w_0,\dots,w_n)\mapsto \text{conjugacy class of semisimple part of $\tMz$}
\end{gather*}
is bijective.
\end{pro}

Thus,
$\cS_{n+1}^{\text{global}}$ can be regarded
as a cross-section of the monodromy map
$\cS_{n+1}^{\text{zero}} \to \cM_{n+1}^{\text{zero}}$.
Another consequence is the following characterization of
$\cS_{n+1}^{\text{zero}}$.  This is analogous to the characterization of
Proposition \ref{pro:characterizationbygamma} using asymptotic data, but now using monodromy data:

\begin{cor}\label{cor:characterizationbyq}  (i) For any
solution $w_i$ defined on an interval of the form $(0,\epsilon)$, such that $w_i=w_i(|t|)$ and $2w_i\sim \gamma_i\log |t|$ as $z\rightarrow 0$, the matrices $\tQz_k$ are real and the eigenvalues of $\tMz$ are in the unit circle.
(ii) Conversely, if $w_i$ is any local solution such that $w_i=w_i(|t|)$, and such that the matrices $\tQz_k$ are real and the eigenvalues of $\tMz$ are in the unit circle, then $w_i$ arises as in (i).
\end{cor}

\begin{proof}  (i) Theorem \ref{thm:eigenvalues} and Corollary \ref{cor:real}. (ii)
The hypotheses imply that $\tMz\in\cR_{n+1}^{\simple,c,p}$.  By Theorem \ref{thm:main}, $\tMz\in\cM_{n+1}^{\text{zero}}$.
\end{proof}

\section{Examples}\label{examples}

In this section, we exhibit concrete matrices for the cases $n+1=4,\ 5$ cases, which were the main focus of \cite{GIL1}\cite{GIL2}.  We use the notation of those articles.

\subsection{The case $n+1=4$}
$  $

The singular directions are $\tfrac\pi4\bZ$.  We choose
$\theta_1=-\tfrac\pi4, \theta_{1\scriptstyle\frac 1{4}}=-\tfrac\pi2, \dots$ and find that the roots associated to (a half-period of) singular directions are as shown in Table \ref{t1example4}.

\begin{table}[h]
\renewcommand{\arraystretch}{1.3}
\begin{tabular}{c|c}
$\theta$  & $\cR(\theta)$
\\
\hline
$\theta_1$ & $\alpha_{1,0},\alpha_{2,3}$
\\
$\theta_{1\scriptstyle\frac 1{4}}$  & $\alpha_{1,3}$
\\
$\theta_{1\scriptstyle\frac 1{2}}$  & $\alpha_{0,3},\alpha_{1,2}$
\\
$\theta_{1\scriptstyle\frac 3{4}}$  & $\alpha_{0,2}$
\end{tabular}
\bigskip
\caption{Roots associated to Stokes factors when $n+1=4$}
\label{t1example4}
\end{table}
The fundamental Stokes factors can be written as
\[
\tQz_1=
\begin{pmatrix}
1 & & & \\
-s^\bR_1 & 1 & & \\
 & & 1 & s^\bR_1 \\
  & & & 1
\end{pmatrix},
\ \
\tQz_{1\scriptstyle\frac14}=
\begin{pmatrix}
1 & & & \\
 & 1 & & -s^\bR_2\\
 & & 1 &  \\
  & & & 1
\end{pmatrix}
\]
where $s^\bR_1,s^\bR_2\in\bR$ (see the discussion after Lemma 2.3 and after Theorem 5.5 in \cite{GIL2}).
The characteristic polynomial of the matrix $\tMz=\tQz_1\tQz_{1\scriptstyle\frac14}\hat\Pi$ is $\mu^4 + s^\bR_1 \mu^3 - s^\bR_2 \mu^2 + s^\bR_1 \mu + 1$, which is real, and palindromic. The eigenvalues of $\tMz$
are $e^{\pm\frac{\gamma_0+1}{4}\pi\i}$, $e^{\pm\frac{\gamma_1+3}{4}\pi\i}$ where
$-1\leq \gamma_0\leq 3, -3\leq\gamma_1\leq 1,\gamma_1-\gamma_0\geq-2$ as stated in Theorem \ref{thm:convex}.  It follows that $s^\bR_1,s^\bR_2\in\bR$ are given explicitly by:
\begin{align*}
s_1^\bR &= -2\cos \tfrac\pi4 {\scriptstyle (\gamma_0+1)} -  2\cos \tfrac\pi4 {\scriptstyle(\gamma_1+3)}
\\
s_2^\bR &= -2-4\cos \tfrac\pi4 {\scriptstyle(\gamma_0+1)} \, \cos \tfrac\pi4 {\scriptstyle(\gamma_1+3)}
\end{align*}

The system of positive roots shown in the table corresponds to the choice of simple roots
$\alpha_{1,0},\alpha_{2,3}, \alpha_{0,2}$, and (in the notation of Proposition \ref{pro:simplerootseven}),  to the diagram

\begin{center}
\begin{tabular}{cccc}
$3$ & &$0$
\\
 & $2$ & &$1$
\end{tabular}
\end{center}

The space $\cM_4$ consists of all matrices of the form
\[
\begin{pmatrix}
 & 1& & \\
 -x_{13} & x_{10} & 1&\\
\ \  x_{23} & &  &  1 \\
 -1 & & &
\end{pmatrix}
\]
with $x_{10}, x_{13}, x_{23} \in\bC$. The
characteristic polynomial of this matrix is $\mu^4-x_{10}\mu^3+x_{13}\mu^2-x_{23}\mu+1$.

Let us verify directly that $\cM_4=\cR^\simple_4$, where $\simple$ denotes the simple roots
\[
\alpha_0=\alpha_{1,0},
\quad
\alpha_1=\alpha_{2,3},
\quad
 \alpha_2=\alpha_{0,2}.
\]
Corresponding Weyl group elements can be represented by
\[{\tiny
\n_0=
\left(\begin{array}{cccc}
0&1&0& 0\\
- 1 &0 &0 & 0\\
0 & 0 &1&0 \\
0&0&0&1
\end{array}\right),\quad
\n_1=
\left(\begin{array}{cccc}
1&0& 0&0\\
0 &1&0  & 0\\
 0 & 0 &0&1  \\
 0&0 & -1 &0
\end{array}\right) ,\quad
\n_2=
\left(\begin{array}{cccc}
0&0&-1& 0\\
0& 1 &0 & 0\\
1&  0&0 &0  \\
0 & 0 &0 &1
\end{array}\right),}
\]
and root vectors by
\[{\tiny
E_{\alpha_0}=\left(\begin{array}{cccc}
0&0&0&0\\
1& 0 & 0&0\\
0&  0 & 0&0 \\
0 & 0 &0 &0
\end{array}\right),\quad
E_{\alpha_1}=\left(\begin{array}{cccc}
0 &0&0  & 0\\
0& 0 & 0 &0  \\
0&0&0&1\\
0 & 0 &0&0
\end{array}\right),\quad
E_{\alpha_2}=\left(\begin{array}{cccc}
0&0& 1&0\\
 0 &0  & 0&0\\
 0 & 0 & 0&0 \\
0&0&0&0
\end{array}\right).}
\]
Recall that $e_i(t_i)=\exp(t_iE_{\alpha_i})=I+t_iE_{\alpha_i}$. Then
\[
e_0(x_{10})\n_0e_1(x_{23})\n_1e_2(x_{13})\n_2
=
\begin{pmatrix}
 & \ \ 1& &  \\
\  -x_{13} & \ \ x_{10}  &\ 1&\\
\ \   x_{23} & &  & \ 1 \\
 \  -1 & &  &
\end{pmatrix}
\]
which is the general element of $\cM_4$,
as asserted, and also $\n_0 \n_1 \n_2 =\hat\Pi$.

\subsection{The case $n+1=5$}
$  $

The singular directions are $\tfrac\pi{10}+\tfrac\pi5\bZ$.  We choose
$\theta_1=-\tfrac\pi{10}, \theta_{1\scriptstyle\frac 1{5}}=-\tfrac{3\pi}{10}, \dots$ and find that the roots associated to (a half-period of) singular directions are as shown in Table \ref{t1example5}.

\begin{table}[h]
\renewcommand{\arraystretch}{1.3}
\begin{tabular}{c|c}
$\theta$  & $\cR(\theta)$
\\
\hline
$\theta_1$ & $\alpha_{2,0},\alpha_{3,4}$
\\
$\theta_{1\scriptstyle\frac 1{5}}$  & $\alpha_{1,0},\alpha_{2,4}$
\\
$\theta_{1\scriptstyle\frac 2{5}}$  & $\alpha_{1,4},\alpha_{2,3}$
\\
$\theta_{1\scriptstyle\frac 3{5}}$  & $\alpha_{0,4},\alpha_{1,3}$
\\
$\theta_{1\scriptstyle\frac 4{5}}$  & $\alpha_{0,3},\alpha_{1,2}$
\end{tabular}
\bigskip
\caption{Roots associated to Stokes factors when $n+1=5$}
\label{t1example5}
\end{table}
The fundamental Stokes factors can be written as
\[
\tQz_1=
\begin{pmatrix}
1 & & & &\\
 & 1 & & & \\
s^\bR_2 & & 1 &  & \\
  & & & 1 & -s^\bR_1\\
  & & & & 1
\end{pmatrix},
\ \
\tQz_{1\scriptstyle\frac15}=
\begin{pmatrix}
1 & & & &\\
 s^\bR_1 & 1 & & & \\
 & & 1 &  & -s^\bR_2 \\
  & & & 1 & \\
  & & & & 1
\end{pmatrix}
\]
where $s^\bR_1,s^\bR_2\in\bR$.
The characteristic polynomial of the matrix $\tMz=\tQz_1\tQz_{1\scriptstyle\frac15}\Pi$ is $\mu^5 - s^\bR_1 \mu^4 - s^\bR_2 \mu^3 + s^\bR_2 \mu^2 +
 s^\bR_1 \mu - 1$, which is real and anti-palindromic. The eigenvalues of $\tMz$
are $e^{\pm\frac{\gamma_0-4}{5}\pi\i}$, $e^{\pm\frac{\gamma_1-2}{5}\pi\i},1$ where
$-1\leq \gamma_0\leq 4, -3\leq\gamma_1\leq 2,\gamma_1-\gamma_0\geq-2$.
 It follows that $s^\bR_1,s^\bR_2\in\bR$ are given explicitly by:
\begin{align*}
s_1^\bR &= 1+2\cos \tfrac\pi5 {\scriptstyle(\gamma_0-4)}  + 2\cos \tfrac\pi5 {\scriptstyle(\gamma_1-2)}
\\
s_2^\bR &= -2
-2\cos \tfrac\pi5 {\scriptstyle(\gamma_0-4)}  -  2\cos \tfrac\pi5 {\scriptstyle(\gamma_1-2)}
-4\cos \tfrac\pi5 {\scriptstyle(\gamma_0-4)} \, \cos \tfrac\pi5 {\scriptstyle(\gamma_1-2)}
\end{align*}

The system of positive roots in the table corresponds to the choice of simple roots
$\alpha_{2,0},\alpha_{3,4}, \alpha_{0,3}, \alpha_{1,2}$, and (in the notation of Proposition \ref{pro:simplerootsodd}),  to the diagram

\begin{center}
\begin{tabular}{ccccc}
$4$ & &$0$ & &1
\\
 & $3$ & &$2$ &
\end{tabular}
\end{center}

The space $\cM_5$ consists of all matrices of the form
\[
\begin{pmatrix}
 & \ \ 1& & & \\
 &\ x_{10}& 1 & & \\
\ x_{24} & \ x_{20} & &\ 1&\\
\  x_{34} && &  & \ 1 \\
 \ \ 1 & & & &
\end{pmatrix}
\]
with
$x_{20},x_{34},x_{10},x_{24} \in \bC$.
The
characteristic polynomial of this matrix is
$\mu^5-x_{10}\mu^4-x_{20}\mu^3-x_{24}\mu^2-x_{34}\mu-1$.

Let us verify directly that $\cM_5=\cR^\simple_5$, where $\simple$ denotes the simple roots
\[
\alpha_0=\alpha_{2,0},
\quad
\alpha_1=\alpha_{3,4},
\quad
\alpha_2=\alpha_{0,3}
\quad
\alpha_3=\alpha_{1,2}.
\]
Corresponding Weyl group elements $\n_0,\n_1,\n_2,\n_3$ can be represented (respectively) by

\[{\tiny
\left(\begin{array}{ccccc}
0&0&1&0& 0\\
0& 1 &0 &0 & 0\\
-1& 0 & 0 &0&0  \\
0&0 & 0 &1&0 \\
0&0&0&0&1
\end{array}\right)\!\!,
\
\left(\begin{array}{ccccc}
-1 &0&0&0&0 \\
0&1&0& 0&0\\
0& 0 &1&0  & 0\\
0& 0 & 0 &0&1  \\
0& 0&0 & 1 &0
\end{array}\right)\!\!,
\
\left(\begin{array}{ccccc}
1&0&0&0& 0\\
0&0&1& 0&0\\
0& 1 &0 &0 & 0\\
0& 0 & 0&1 &0  \\
0&0 & 0 &0 &-1
\end{array}\right)\!\!,
\
\left(\begin{array}{ccccc}
0 &0&0&1&0 \\
0&1&0& 0&0\\
0& 0 &1&0  & 0\\
1& 0 & 0 &0&0  \\
0& 0&0 & 0 &-1
\end{array}\right)}
\]

\noindent
and root vectors $E_{\alpha_0},E_{\alpha_1},E_{\alpha_2},E_{\alpha_3}$ (respectively) by

\[{\tiny
\left(\begin{array}{ccccc}
0&0&0& 0&0\\
0&0&0&0&0\\
1& 0 &0  & 0&0\\
0& 0 & 0 & 0&0 \\
0&0 & 0 &0 &0
\end{array}\right)\!\!,
\
\left(\begin{array}{ccccc}
0&0& 0&0& 0\\
0& 0 &0&0  & 0\\
0&0& 0 & 0 &0  \\
0&0&0&0&1\\
0&0 & 0 &0&0
\end{array}\right)\!\!,
\
\left(\begin{array}{ccccc}
0&0&0& 1&0\\
0& 0 &0  & 0&0\\
0& 0 & 0 & 0&0 \\
0&0 & 0 &0 &0\\
0&0&0&0&0
\end{array}\right)\!\!,
\
\left(\begin{array}{ccccc}
0&0&0&0&0\\
0&0& 1& 0&0\\
0& 0 &0  & 0&0\\
0& 0 & 0 &0  &0\\
0&0 & 0 &0&0
\end{array}\right)\!.}
\]

We obtain
\[
e_0(x_{20})\n_0e_1(x_{34})\n_1e_2(x_{10})\n_2e_3(x_{24})\n_3
=
\begin{pmatrix}
 & \ \ 1& & & \\
 & \ \ x_{10}& 1 & & \\
\ \  x_{24} & \ \ x_{20}  & &\ 1&\\
\ \   x_{34} && &  & \ 1 \\
 \ \ 1 & & & &
\end{pmatrix}
\]
which is the general element of $\cM_5$,
as asserted, and also $\n_0 \n_1 \n_2 \n_3 =\Pi$.

\section{Proof of Theorem \ref{thm:eigenvalues}}\label{app}

For the cases $n+1=4,5$ the proof can be extracted from section 4 of \cite{GIL1} and section 5 of \cite{GIL2}. For the case $n+1=4$ another proof is given in section 3 of \cite{GIL3}, and we shall show how to extend that proof to the case of any $n$.   Another proof for any $n$ can be extracted from \cite{M1}\cite{M2}.

The following proof is based on the consideration of a meromorphic connection
\[
\hat\omega=
\left(
-\tfrac{n+1}N \tfrac{z}{\lambda^2}
\ \eta
+\tfrac1\lambda
\ m
\right)
d\lambda
\]
where
\[
\eta=
\bp
 & & & p_0\\
 p_1 & & & \\
  & \ddots & & \\
   & & p_n &
\ep,
\quad
m=
\bp
m_0 & & \\
 & \ddots & \\
  & & m_n
  \ep.
\]
Here $z$ is a complex parameter and
$p_i=c_i z^{k_i}$, where $c_i>0, k_i\ge -1$, and
$N=n+1+\sum_{i=0}^\infty k_i$.  The real numbers $m_i$ are defined in terms of the $k_i$ by
\begin{align*}
\tfrac{n+1}N(k_0+1)&=1+m_n-m_0\\
\tfrac{n+1}N(k_1+1)&=1+m_0-m_1\\
&\dots \\
\tfrac{n+1}N(k_n+1)&=1+m_{n-1}-m_n\\
\end{align*}
together with the conditions $m_i+m_{n-i}=0$.

It is explained in section 2 of \cite{GIL3} how the
connection $\hat\omega$ is related to the  connection $\hat\alpha$ by a loop group Iwasawa factorization argument and the variable transformations
$t=\tfrac{n+1}N \ c^{\frac1{n+1}} \, z^{\frac N{n+1}}$, $\zeta=\lambda/t$,
where $c=c_0\dots c_n$.
It follows from this that
\begin{equation}\label{mandgamma}
m_i=-\tfrac12 \gamma_i.
\end{equation}

The leading terms
$-\frac {n+1}N\frac{z}{\lambda^2}\eta\, d\lambda$ of $\hat\omega$
and
$-\frac 1{\zeta^2} W^T d\zeta$ of $\hat\alpha$
are
conjugate, and the Stokes analysis of $\hat\omega$ can be carried out in exactly the same way as that of $\hat\alpha$, using the same Stokes sectors. This gives rise to a canonical solution $\Phiz_k$ of the system
\begin{equation}\label{meromorphicsystem2}
\Phi_\lambda =
\left(
-\tfrac{n+1}N \tfrac{z}{\lambda^2}
\ \eta^T
+\tfrac1\lambda
\ m
\right)\Phi
\end{equation}
on  the sector $\Omz_k$.

Furthermore, $\hat\omega$ satisfies

\noindent{\em Cyclic symmetry: }  $\tau(\hat\omega(\lambda))
=\hat\omega(\omega \lambda)$

\noindent{\em Anti-symmetry: }  $\sigma(\hat\omega(\lambda))=\hat\omega(-\lambda)$

\noindent (as for $\hat\alpha$, but excluding reality).
From the cyclic symmetry (Lemma 3.1 of \cite{GIL3}) we obtain
\begin{equation}\label{phizerocyclic}
d_{n+1}^{-1}\Phiz_{k}(\omega\lambda)\Pi^{-1}=\Phiz_{k\scriptstyle\frac2{n+1}}(\lambda)
= \Phiz_{k}(\lambda) \Qz_k\Qz_{k\scriptstyle\frac1{n+1}}.
\end{equation}

At $\lambda=\infty$, the system (\ref{meromorphicsystem2}) has a regular singular point, and thus a canonical solution $\Phii$ (in the sense of Theorem 1.2 of \cite{FIKN06}).  The cyclic symmetry (Lemma 3.4 of \cite{GIL3}) is
\begin{equation}\label{phiinfinitycyclic}
d_{n+1}^{-1}\Phii(\omega\lambda)d_{n+1}\omega^{-m}=\Phii(\lambda).
\end{equation}

The solutions $\Phii,\Phiz_k$ must satisfy $\Phii=\Phiz_k D_k$ for some invertible (constant) matrix $D_k$.  Comparison of (\ref{phizerocyclic}) and
(\ref{phiinfinitycyclic}) now gives:
\[
\Qz_1\Qz_{1\scriptstyle\frac1{n+1}}\Pi = D_k d_{n+1}^{-1} \omega^m D_k^{-1}
\]
from which we see that the eigenvalues of $\Qz_1\Qz_{1+\scriptstyle\frac1{n+1}}\Pi$
are the same as those of the diagonal matrix $d_{n+1}^{-1} \omega^m$.

Strictly speaking, this argument holds only when the regular singularity is non-resonant.  In the resonant case $d_{n+1}^{-1} \omega^m$ must be replaced by
$d_{n+1}^{-1} \omega^m \omega^M$ where $M$ is nilpotent (section 6 of \cite{GIL3}).  However this does not change the eigenvalues.

It remains to convert from the $\Qz_k$ of \cite{GIL3} to the  $\tQz_k$ of \cite{GIL2}. The relation is
\[
\tQz_k = d_{(0)}^{-1} \Qz_k d_{(0)}
\]
where
\[
d_{(0)}=
\begin{cases}
\diag(1,\omega^{\frac12},\omega^{1},\omega^{\frac32},\dots,\omega^{\frac{n+1}2})
\quad\text{if $n+1=2m$}
\\
\diag(1,\omega^{m+1},\omega^{2(m+1)},\dots,\omega^{n(m+1)})
\quad\text{if $n+1=2m+1$}
\end{cases}
\]
From this we obtain
\begin{align*}
\tQz_1 \tQz_{1\scriptstyle\frac1{n+1}}\hat\Pi
&=\omega^{-\frac12} \Qz_1 \Qz_{1\scriptstyle\frac1{n+1}}\Pi
\quad\text{if $n+1=2m$}
\\
\tQz_1 \tQz_{1\scriptstyle\frac1{n+1}}\Pi
&=
\omega^{-(m+1)} \Qz_1 \Qz_{1\scriptstyle\frac1{n+1}}\Pi
\quad\text{if $n+1=2m+1$}
\end{align*}
This gives the required conclusion.  Namely, if $n+1=2m$, the eigenvalues of $\tQz_1 \tQz_{1\scriptstyle\frac1{n+1}}\hat\Pi$ are those of the diagonal matrix
$\omega^{-\frac12}d_{n+1}^{-1} \omega^m$, and, $n+1=2m+1$, the eigenvalues of
$\tQz_1 \tQz_{1\scriptstyle\frac1{n+1}}\Pi$ are those of the diagonal matrix
$\omega^{-(m+1)} d_{n+1}^{-1} \omega^m$. The expressions involving the $\gamma_i$ in
Theorem \ref{thm:eigenvalues} are obtained by using (\ref{mandgamma}).

\end{document}